\documentclass[12pt,a4paper]{amsart}

            		\usepackage{amsmath, graphicx, cite, enumerate, comment}
            		\usepackage{amssymb, amsthm, amscd}
            		\usepackage{latexsym}
            		\usepackage{amssymb}
            		\usepackage[T1]{fontenc}	
            		\usepackage[english]{babel}
            		\usepackage{amsmath,amssymb}
            		\usepackage{amsfonts}
            		\usepackage{amscd}
            		\usepackage{cite}
            		\usepackage{marvosym}
            		\usepackage{mathrsfs}
            		\usepackage{amsthm}
            		\usepackage{hyperref}
            		\usepackage{verbatim}
            		\usepackage{color}

            		\hypersetup{
            			bookmarks=true,         
            			unicode=false,          
            			pdftoolbar=true,        
            			pdfmenubar=true,        
            			pdffitwindow=false,     
            			pdfstartview={FitH},    
            			pdftitle={My title},    
            			pdfauthor={Author},     
            			pdfsubject={Subject},   
            			pdfcreator={Creator},   
            			pdfproducer={Producer}, 
            			pdfkeywords={keyword1, key2, key3}, 
            			pdfnewwindow=true,      
            			colorlinks=true,       
            			linkcolor=black,          
            			citecolor=blue,        
            			filecolor=magenta,      
            			urlcolor=cyan           
            		}

            		\usepackage{tikz} \usetikzlibrary{arrows,shapes,patterns}
            		\usetikzlibrary{calc}
            		\usetikzlibrary{intersections}
            		\usepackage{lipsum} 

            		\newtheorem{thm}{Theorem}
            		\newtheorem{lem}[thm]{Lemma}
            		
            		\newtheorem{prop}[thm]{Proposition}

            		\newtheorem{rmk}[thm]{Remark}

            		\theoremstyle{definition}
            		\newtheorem{defi}[thm]{Definition}

            		\numberwithin{equation}{section}

            		\usepackage{soul}
            		\setulcolor{gray}
            		\setul{1pt}{1pt} 

            		\newcommand{\be}{\begin{equation}}
            		\newcommand{\ee}{\end{equation}}
            		\newcommand{\bd}{\begin{displaymath}}
            		\newcommand{\ed}{\end{displaymath}}

\begin{document}
            			
            			\date{}
            			\title[Min-max embedded geodesic lines]
            			{Min-max embedded geodesic lines \\ in asymptotically conical surfaces}
            			\author{Alessandro Carlotto}
            			\address{ETH - Department of Mathematics  \\
            				ETH \\
            				Z\"urich, Switzerland}
            			\email{alessandro.carlotto@math.ethz.ch}
            			\author{Camillo De Lellis}
            			\address{Institut f\"ur Mathematik \\
            				Universit\"at Z\"urich \\
            				Z\"urich, Switzerland}
            			\email{camillo.delellis@math.uzh.ch}           
            			
            			\begin{abstract}
            				We employ min-max methods to construct uncountably many, geometrically distinct, properly embedded geodesic lines in any asymptotically conical surface of non-negative scalar curvature, a setting where minimization schemes are doomed to fail. Our construction provides control of the Morse index of the geodesic lines we produce, which will be always less or equal than one (with equality under suitable curvature or genericity assumptions), as well as of their precise asymptotic behaviour. In fact, we can prove that in any such surface for every couple of opposite half-lines there exists an embedded geodesic line whose two ends are asymptotic, in a suitable sense, to those half-lines.
            			\end{abstract}
            			
            			\maketitle
            			
            			\section{Introduction}
            			
            			The quest for closed geodesics in compact Riemannian manifolds has been one of the main themes in the modern history of differential geometry. This problem, dating back at least to H. Poincar\'e \cite{Poi05}, has been approached by a variety of methods, whose development turned out to be remarkably useful in a wide range of fields. Among these, special importance is deserved by the curve-shortening scheme proposed by G. Birkhoff \cite{Bir17} in order to construct simple closed geodesics on manifolds whose fundamental group is trivial, so that minimization methods are not successfully applicable. The ideas behind this approach turned out to be crucial in the development of min-max methods for the area functional, which allowed Almgren and Pitts \cite{Pit81} to show existence of at least one closed embedded minimal hypersurface in any compact manifold of dimension less than six (which was later extended to higher dimensions by Schoen and Simon \cite{SS81}). In turn, these methods proved to be extremely powerful in tackling a number of fundamental questions in geometry, like the Willmore conjecture \cite{MN14}, the Freedman conjecture on the energy of links \cite{AMN16}, the Yau conjecture on minimal hypersurfaces in manifolds of positive Ricci curvature \cite{MN16} and (most recently) the problem of constructing new classes of (higher genus) solitons for the mean curvature flow \cite{Ket16}, just to name a few. We refer the reader to the beautiful ICM lectures by F. Marques \cite{Mar14} and A. Neves \cite{Nev14} for a broader overview and contextualization of these methods. \\
            			
            			\indent A somehow analogous question, also of global nature, is that of existence of \ul{embedded} geodesic lines: given a (non-compact) Riemannian manifold $(M,g)$ we wonder about the existence of (proper) geodesic embeddings $\gamma:\mathbb{R}\to M$. In general, the answer to such a question depends, in a dramatic fashion, on the topology and on the asymptotic structure of $(M,g)$. Simple existence theorems are only at disposal when the problem is approachable by minimization methods, that is to say by taking limits of length-minimizing geodesic segments for endpoints escaping at infinity in some appropriate fashion.
            			This approach does indeed work, for instance, if $(M,g)$ has suitable convexity properties at infinity. When these sorts of assumptions are not made, trying to construct geodesic lines by solving a sequence of minimization problems will not work as there are in general no geometric reasons for the sequence of geodesic segments one may construct not to escape from any given compact subset of the manifold in question. In fact, this is the typical behaviour on positively curved manifolds for in that case the formula for the second variation of the length functional ensures that no stable geodesic lines can actually exist. An important class of surfaces that exhibit this phenomenon is provided by asymptotically conical ones, which arise as asymptotically flat models in $2+1$ gravity (see e. g. \cite{BBL86, Bro88, Col88, DJtH84} and \cite{Car98}). In that context the non-negativity of the scalar curvature is just a reflex of the dominant energy condition (DEC) and the non-existence of embedded, stable geodesic lines is the two-dimensional counterpart of a well-known obstruction disclosed by Schoen-Yau in their proof of the Positive Mass Theorem \cite{SY79} and recently widely investigated in its connections to the large-scale structure of isolated gravitational systems \cite{EM12, EM13, Car14, CCE15, CS14}.\\
            			
            			\indent In spite of the fact that minimization methods are doomed to fail, we shall prove here that every asymptotically conical surface does in fact contain lots of (properly) embedded geodesic lines, whose Morse index equals one under natural curvature conditions.
            			
            			\begin{thm}\label{thm:maingeod}
            				An asymptotically conical surface of non-negative scalar curvature contains infinitely many, geometrically distinct, properly embedded geodesic lines of Morse index less or equal than one. If the scalar curvature is assumed to be positive equality holds.
            			\end{thm}	
            			
            			We refer the reader to Subsection \ref{subs:acs} for a precise definition of asymptotically conical surface and for the recollection of a few basic facts. A brief discussion of the positive mass theorem in two spatial dimensions is provided in Subsection \ref{subs:pmt2}. \\
            			
            			\indent In order to avoid dangerous misunderstandings, let us remark here that the geodesic lines we construct are \textsl{never} length-minimizing (in other terms: they are not \textsl{straight lines}) for otherwise the ambient manifold $(S,g)$ would of course split as a Riemannian product (by virtue of the well-known theorem by S. Cohn-Vossen \cite{CV36}, later exetended by Cheeger-Gromoll \cite{CG71}). In fact, our result ensures that complete, embedded geodesics that are not length-minimizing exist in abundance under very natural assumptions on the asymptotic behaviour of the ambient manifold.

            			\begin{rmk} In the statement of Theorem \ref{thm:maingeod} the assumption that the scalar curvature be non-negative forces the surface in question to be a complete plane (namely: to have only one end and genus zero), see Theorem \ref{lem:pmt2}. However, this is not restrictive (as far as one is concerned with the existence problem) for in the presence of at least two ends one can just obtain properly embedded geodesic lines by taking a limit of minimizing segments whose endpoints diverge on different ends of the surface in question. In fact, a similar strategy allows to deal with the case when the surface contains a non-separating closed curve and thus solves the problem when the genus is not zero (see e. g. \cite{Ban81a} pg. 64). 
            			\end{rmk} 
            			
            			
            			\indent It is important to contextualize our result with respect to the rich history concerning the quest for \textsl{escaping rays}. For a broad overview, with several remarkable contributions, of the study of maximal geodesics on complete surfaces we refer the reader to \cite{Shi94, SST03} and references therein. 
            			The question of existence of proper geodesic embeddings $\gamma:\mathbb{R}\to M$ for $M$ a Riemannian manifold homeomorphic to the plane was explicitly posed by S. Cohn-Vossen in 1936. After various sorts of partial contributions, it was then finally solved by V. Bangert in 1981 (see \cite{Ban81a}, as well as \cite{Ban81b,Ban81c} and references therein for related and ancillary results; see also Proposition 6.1 in \cite{BL03} for a refined existence result for planes of finite positive total curvature). Yet, the arguments employed to answer such a question in full generality are rather indirect and provide little information on the geodesic line whose existence is proven and in particular do not provide any information at all about the Morse index of the line itself. The author needs to distinguish various cases, depending on whether the surface does contain a closed geodesic or not. In this respect, Bangert states (in \cite{Ban81a}, p. 59): \textsl{We have not been able to find a general method to construct escaping geodesics without self-intersection.} In this paper, we present an effective geometric construction in the category of asymptotically conical surfaces, a category naturally arising in the physical setting described above. Perhaps more importantly, Bangert concluded his article with the following question:
            			\
            			
            			\begin{center}
            				\textsl{Are \kern-0.1em there \kern-0.1em infinitely \kern-0.1em many \kern-0.1em escaping \kern-0.1em geodesics \kern-0.1em on \kern-0.1em every complete \kern-0.1em plane \kern-0.1em $S$?}
            			\end{center}
            			
            			\
            			
            			While very exhaustive results have been achieved in the case of \textsl{negatively curved planes} (see \cite{FM01}), the problem is still far from being completely understood for what concerns positively curved metrics on $\mathbb{R}^2$.
            			Our work provides a novel contribution in this direction, since in fact the argument we describe in Section \ref{sec:bangert} shows that for every couple of opposite rays on our surface we can exhibit an embedded geodesic line whose two ends are asymptotic, in a suitable sense which we shall describe later, to those half-lines. \\

            			\indent The proof of Theorem \ref{thm:maingeod} naturally splits into two parts, the full conclusion following at once by combining Proposition \ref{prop:existence} (for the \textsl{existence} part) and Proposition \ref{prop:multiplicity} (for the \textsl{multiplicity} part). In the next two paragraphs we shall briefly outline them.\\
            			
            			\indent The geodesic lines we construct are obtained by min-max methods. More precisely, we exploit the information on the asymptotic behaviour of our ambient surface to set-up a sequence of Plateau min-max problems (for geodesics) and then check that the sequence of geodesics with boundary we obtain cannot drift off to infinity together with their boundary points. Let us now describe the structure of our proof in more detail. For the first step (which is done in Section \ref{sec:finscale}), we prove that whenever one can join two points on a surface by means of two embedded geodesics that bound a disk \textsl{and} the standard mountain-pass condition holds then there is a third embedded\footnote{The emphasis here is on the word \textsl{embedded} both in the assumption and in the conclusion of our assertion, for otherwise the result would just be a routine application of one-dimensional min-max schemes.}geodesic joining the two points in question (Proposition \ref{prop:finscale}). Of course, such a condition is automatically satisfied when the two geodesics that are given are strictly stable. This result, of independent interest and potentially wide applicability, relies on the combined use of the \textsl{one-dimensional} $H^1$-gradient flow and, perhaps more importantly, on the clever \textsl{resolution of singularities} procedure proposed by G. Chambers and Y. Liokumovich \cite{CL14} in order to effectively convert homotopies into isotopies. For the scope of controlling the index, we have found it convenient to work with the energy functional (rather than the length functional), somehow in the spirit of the parametric approach to the min-max Plateau problem proposed long time ago by Shiffman \cite{Shi39} and Morse-Tompkins \cite{MT39}. The fact that the embedded geodesic segments that we produce do not escape from any given compact set is proven by using the Gauss-Bonnet theorem and a blow-down procedure (since \textsl{all} geodesics connecting two antipodal points at the same height on a cone are known).
            			This \textsl{no-drift} argument, which lies at the core of this construction, is presented in Section \ref{sec:proof}. 
            			
            			At that stage, we prove that this method does in fact produce \textsl{uncountably many} distinct geodesic lines. Roughly speaking, this is achieved as follows. The min-max geodesic segments we produce connect (by construction) couple of \textsl{antipodal} points in the asymptotic region, where one has coordinates $(r,\varphi)\in (0,+\infty)\times S^1$ and each geodesic line is obtained as a (subsequential) geometric limit as the first coordinate of such points goes to infinity, with the second coordinate fixed to values $\varphi_0$ and $\varphi_0+\pi$. In principle, one expects that as we vary $\varphi_0$ we should indeed obtain distinct geodesic lines, but this is not obvious as \textsl{twisting phenomena} may occur without leading to any contradiction by means of a direct blow-down procedure. The relevant argument, which proves Proposition \ref{prop:multiplicity}, is described in Section \ref{sec:bangert}. \\

            			\indent When considering our work in the context of min-max techniques, one direction we should mention is the development of methods for contructing closed (or, more generally, finite area) minimal hypersurfaces in non-compact Riemannian manifolds, due to Bangert \cite{Ban80} and Thorbergsson \cite{Tho78} for the special case of finite length geodesics on certain non-compact surfaces (of finite area) and, much more recently, remarkably extended by Montezuma \cite{Mon14} (resp. Chambers-Liokumovich \cite{CL16}) to handle closed (resp. finite area) minimal hypersurfaces in classes of manifolds satisfying various types of asymptotic geometric conditions. Our scope here is rather different: while our setting is also non-compact, we aim at constructing variational objects which are themselves non-compact and for which the relevant functional (in our case: the energy functional) attains infinite value, thereby providing an obstacle of new and peculiar nature. In fact, the next step in our programme is precisely the extension of the methods presented here in the special case of geodesics to the construction of complete (non-compact) minimal hypersurfaces in suitable classes of non-compact Riemannian manifolds.     \\

            			\textsl{Acknowledgments}. This project was partly developed while both authors where at Harvard University as visiting scholars at the Center of Mathematical Sciences and Applications on invitation by Shing-Tung Yau: the warm hospitality and excellent working conditions provided by these institutions are gratefully acknowledged. The first-named author would like to thank Richard Schoen for a number of enlightening conversations on themes related to min-max constructions in non-compact manifolds. The authors also wish to express their sincere gratitude to the anonymous referees for carefully reading this article and for preparing detailed reports, which resulted in a significantly improved version of the paper.
            			 This article was prepared while the first-named author was an ETH-ITS fellow. The second-named author acknowledges the support of the Swiss National Foundation, through grant SNF 159403.

            			\section{Setting and recollections}\label{sec:setup}
            			
            			\subsection{Cones and their geodesics}
            			
            			Let us consider on $\mathbb{R}^2\setminus\left\{0\right\}$ the local parametrizations obtained by restrictions of the smooth covering map $F:(0,+\infty)\times S^1\to \mathbb{R}^2\setminus\left\{0\right\}$ defined by
            			\[
            			F(r,\varphi)=(r\cos\varphi, r\sin\varphi).
            			\]
            			For $\alpha\in (0,\pi/2]$ we consider on $\mathbb{R}^2\setminus\left\{0\right\}$ the incomplete Riemannian metric
            			\[
            			g_\alpha = dr\otimes dr + r^2\sin^2(\alpha)d\varphi\otimes d\varphi
            			\]
            			and we let $C_{\alpha}$ denote the corresponding Riemannian manifold $(\mathbb{R}^2\setminus\left\{0\right\}, g_{\alpha})$. We shall also consider the (metric) closure $\overline{C}_{\alpha}$, a complete (singular) cone of opening angle\footnote{In order to avoid ambiguities, let us remark that $\alpha$ is the angle between the axis and the generatrix of the cone $C_{\alpha}$ when this is isometrically embedded in $\mathbb{R}^3$ in the standard fashion.}$\alpha$. We let $d$ denote the (path)-distance on $\overline{C}_{\alpha}$ and $v\in\overline{C}_{\alpha}\setminus C_{\alpha}$ the vertex of the cone. \\
            			
            			\indent In order to perform our min-max conctruction, we need to recall a few facts, whose proofs are straightforward consequences of the characterization of geodesics in flat $\mathbb{R}^2$. 
            			
            			\begin{defi}  We will say that two points $p, q \in C_{\alpha}$ are antipodal if in the coordinate charts above one has $r(p)=r(q)$ and $|\varphi(p)-\varphi(q)|=\pi$.
            			\end{defi}
            			
            			\begin{lem}\label{lem:geodcon}
            				In the setting above, when $\alpha\in (0,\pi/2)$ for every couple of antipodal points $p, q$ (set $r:=r(p)=r(q)$) there are exactly two smooth, distinct geodesics connecting them (whose length equals $2r\sin\left(\frac{\pi}{2}\sin\alpha\right)$) and one singular geodesic (whose length equals $2r$). In particular
            				\[
            				d(p,q)= 2 d(p,v)\sin\left(\frac{\pi}{2}\sin\alpha\right)=2 d(q,v)\sin\left(\frac{\pi}{2}\sin\alpha\right).
            				\]
            				
            			\end{lem}
            			
            			
            			We shall also remind the reader of the following important consequence of the Clairaut equation expressing the conservation of angular momentum for geodesics on surfaces of revolution.
            			
            			\begin{lem}\label{lem:clairaut}
            				A (smooth) geodesic on $\overline{C}_{\alpha}$ that intersects every neighborhood of the vertex $v\in\overline{C}_{\alpha}\setminus C_{\alpha}$ must be radial. In other words, if a geodesic path $\gamma:(0,1)\to C_{\alpha}$, parametrized by arc-length, satisfies $-1<g_{\alpha}(\dot{\gamma},\partial_{r})<1$ at some point, then there exists $d_0>0$ such that $\gamma(0,1)$, the support of $\gamma$, is disjoint from the metric ball $B_{d_0}(v)$ on $\overline{C}_{\alpha}$.
            			\end{lem}

            			For a fixed couple of antipodal points at unit distance from the vertex (namely: $d=1$) we shall denote by $\Gamma', \Gamma''$ the geometric support of  the two smooth connecting geodesics and by $\Gamma'''$ the geometric support of the singular geodesic passing through the vertex of the cone.
            			
            			\begin{rmk}\label{rem:flatwedge} Let $l$ be a linear ray in $C_{\alpha}$, that is to say a subset of the form $\left\{\varphi=\overline{\varphi}\right\}$ for some fixed $\overline{\varphi}\in S^1$. One can then consider on $C_{\alpha}\setminus\left\{l\right\}$ the standard planar polar coordinates $(\rho,\vartheta)\in (0,+\infty)\times (0,2\pi\sin\alpha)$ which are obtained by unfolding $C_{\alpha}\setminus\left\{l\right\}$ on a (flat) plane. In particular, the associated map $G:(0,+\infty)\times (0,2\pi\sin\alpha) \to C_{\alpha}$ is in fact an isometry. Furthermore, we can identify the whole $C_{\alpha}$ with the Euclidean wedge $(0,+\infty)\times [0,2\pi\sin\alpha]$ after pointwise identifying the two edges.
            			\end{rmk}
            			
            			\subsection{Asymptotically Conical Surfaces} \label{subs:acs}
            			
            			\begin{defi}\label{def:acon}
            				A complete (non-compact) surface $(S,g)$ is called asymptotically conical if there exists a compact set $Z\subset S $, and a diffeomorphism $\Phi: S\setminus Z \to \mathbb{R}^2\setminus \left\{0\right\}$ such that (endowed $\mathbb{R}^2\setminus \left\{0\right\}$ with coordinates $(r,\varphi)$ as above)
            				\begin{multline*}
            				(\Phi^{-1})^{\ast}g= \left(1+e_{rr}(r,\varphi)\right)dr\otimes dr+ \left(1+e_{\varphi\varphi}(r,\varphi) \right)r^2\sin^2(\alpha)d\varphi\otimes d\varphi \\
            				+ 2e_{r\varphi}(r,\varphi)rdr\otimes d\varphi
            				\end{multline*}
            				for a symmetric $(0,2)$-tensor $e$ satisfying
            				\[
            				e_{rr}(r,\varphi)=O_2(r^{-\mu}), \ e_{\varphi\varphi}(r,\varphi)=O_2(r^{-\mu}), \ e_{r\varphi}(r,\varphi)=O_2(r^{-\mu})
            				\]
            				as we let $r\to+\infty$.
            				We call $\alpha\in (0,\pi/2]$ the asymptotic angle and $\mu>0$ the asymptotic decay rate of the surface $(S,g)$.
            			\end{defi}	
            			
            			\begin{rmk}\label{rem:wedgecoor1}
            				When writing $e(r,\varphi)=O_2(r^{-\mu})$ we mean that
            				\[
            				\partial^{\beta}e(r,\varphi)=O(r^{-\mu-|\beta|_r}), \ \ r\to+\infty
            				\]	
            				for any multi-index $\beta$ such that $0\leq |\beta|\leq 2$ and for $|\beta|_{r}$ equal to the number of differentiations in the variable $r$.
            			\end{rmk}

            			\begin{defi}In the setting of the above definition we will call the couple $(r,\varphi)$ asymptotically conical coordinates for $(S,g)$. Fixing such a structure at infinity, we shall say that two points $p, q\in S\setminus Z$ are antipodal if there exist $(r_0,\varphi_0)$ such that
            				\[
            				(r,\varphi)(p)=(r_0,\varphi_0) \ \textrm{and} \ (r,\varphi)(q)=(r_0,\varphi_0+\pi).
            				\]
            			\end{defi}

            			\begin{rmk}\label{ref:wedgecoor2}
            				in Section \ref{sec:bangert}, it will be convenient to work with \textsl{wedge} coordinates for an asymptotically conical surface $(S,g)$. These are defined in analogy with Remark \ref{rem:flatwedge} and are uniquely determined, once a structure at infinity $(r,\varphi)$ is assigned, by means of the equations
            				$\rho=r,  \vartheta=\varphi \sin\alpha. $
            			\end{rmk}

            		\noindent	The following assertion is a straightforward consequence of Definition \ref{def:acon}.
            			
            			\begin{lem}
            				Given an asymptotically conical surface $(S,g)$ of asymptotic angle $\alpha$ and fixed a structure at infinity $(r,\varphi)$ we consider for a positive parameter $\lambda$ the rescaled metric
            				\[
            				g^{(\lambda)}(r,\varphi)=\lambda^{-2}\textbf{Dil}^{\ast}_{\lambda}g\left( r,\varphi\right)
            				\]	
            				where $\textbf{Dil}_{\lambda}:\mathbb{R}^2\setminus\left\{0\right\}\to\mathbb{R}^2\setminus\left\{0\right\}$ is the diffeomorphism defined by $\textbf{Dil}_{\lambda}(r,\varphi)=(\lambda r,\varphi)$.
            				Then: given any sequence $\left\{\lambda_n\right\}$ such that $\lambda_n\uparrow +\infty$ the sequence of Riemannian manifolds $\left(\mathbb{R}^2\setminus\left\{0\right\}, g^{(\lambda_n)}\right)$ converges (locally uniformly in the $C^2$-topology) to the cone $C_{\alpha}$.
            			\end{lem}

            			This lemma characterizes the blow-down limits of minimizing geodesics connecting antipodal points on asymptotically conical surfaces. The relevant notion of convergence is presented in Appendix \ref{sec:geodcurr}: note that, although such notion of convergence allows for multiplicities higher than one in the limit, in our particular case the latter phenomenon is ruled out by Lemma \ref{lem:convint}.
            			
            			\begin{lem}\label{lem:limgeod}(Notations as above). Let $(S,g)$ be an asymptotically conical surface of asymptotic angle $\alpha\in(0,\pi/2)$.
            				\begin{enumerate}
            					\item{For any couple of antipodal points $p, q\in S\setminus Z$ there exists a length-minimizing geodesic $\Gamma$ connecting them.} 
            					\item{Given a sequence of antipodal points $p^{(k)}, q^{(k)}$ with \[
            						r_k:=r(p^{(k)})=r(q^{(k)})\to+\infty
            						\] and denoted by $\Gamma_k$ the support of a length-minimizing geodesic connecting them, then $\left\{\Gamma_k\right\}$ converges geometrically to either $\Gamma'$ or $\Gamma''$ as we rescale by the sequence $\left\{r_k\right\}$. As a result, the sequence of lengths of rescaled $\Gamma_k$ converges to $2\sin\left(\frac{\pi}{2}\sin\alpha\right)$.}
            				\end{enumerate}
            			\end{lem}	
            			
            			\begin{proof}
            				To prove the first assertion, let us start by observing that there exists a connecting path of length equal half of the circle of coordinate equation $r=r(p)=r(q)$ as we let the variable $\varphi$ vary in an interval of size $\pi$: thus such path has length bounded from above by a fixed constant $C>1$ (depending on the part $e$ of the metric $g$) times $\pi r \sin\alpha$. It follows that any sequence of paths connecting $p$ to $q$ and minimizing length has to be contained inside the coordinate ball of radius  $2C\pi r$. Hence direct methods ensure the existence of such a minimizer. Furthermore, let us explicitly notice that a trivial length comparison argument ensures that the support of $\Gamma$ must be disjoint from the coordinate ball of radius $r_{in}=\left(1-\delta\sin\left(\frac{\pi}{2}\sin\alpha\right)\right)r$, at least for $r$ large enough, for $\delta>1$ chosen once and for all so that $\left(1-\delta\sin\left(\frac{\pi}{2}\sin\alpha\right)\right)>0$.  \\
            				\indent For the second part: as we rescale and take the limit, thanks to the last remark, a standard variation\footnote{For the sake of clarity, we stated the convergence results in Appendix \ref{sec:geodcurr} with respect to a fixed background metric, but the same conclusion does hold true if the manifold $N$ is endowed with a sequence $\left\{g_k\right\}$ of Riemannian metrics that are smoothly converging, uniformly on compact sets of the ambient manifold.}of Lemma \ref{lem:convint} ensures that the sequence $\left\{\Gamma_k\right\}$ will geometrically subconverge to a geodesic on $C_{\alpha}$ hence (by virtue of Lemma \ref{lem:geodcon}) either to $\Gamma'$ or to $\Gamma''$, which completes the proof.
            			\end{proof}

            			\subsection{Positive Mass Theorem in 2+1 gravity}\label{subs:pmt2}
            			
            			As anticipated in the introduction, it is customary in $2+1$ gravity to call \textsl{mass} of an asymptotically conical\footnote{Notice that we could legitimately call this class of spaces \textsl{asymptotically flat}, coherently with the higher-dimensional terminology.}surface $(S,g)$ the angle defect for parallel transport around the limit cone to which $(S,g)$ asymptotes at infinity, namely we shall set
            			\[
            			m= 2\pi(1-\sin\alpha).
            			\]
            			if $\alpha$ is the asymptotic angle of $(S,g)$ in the sense of Definition \ref{def:acon}.
            			This can be fully justified in the context of the Hamiltonian formulation of General Relativity in $2+1$-dimensions, following the same conceptual scheme described by Arnowitt-Deser-Misner \cite{ADM59} when dealing with at least three spatial dimensions (see \cite{Bar86} for a mathematical discussion of the well-posedness of this notion). We refer the reader to Chapter 1 of the lectures by P. Chru\'sciel \cite{Chr10} for a modern, broad treatment of these topics. In that context, we remind the reader that the assumption that the scalar curvature be non-negative is nothing but the aforementioned \textsl{dominant energy condition} (see e. g. \cite{Wal84}).
            			
            			\begin{thm}\label{lem:pmt2}
            				Let $(S,g)$ be an asymptotically conical surface of non-negative scalar curvature. Then
            				$S$ is diffeomorphic to $\mathbb{R}^2$ and, furthermore, $m=0$ if and only if $(S,g)$ is isometric to the Euclidean plane.
            			\end{thm}
            			
            			We present the (easy) proof of this result both for the sake of completeness and due to the absence (to our knowledge) of a standard reference. Yet, we shall remark that the first assertion follows at once from Theorem 1 in \cite{CG72} (such assertion for surfaces being in fact due to S. Cohn-Vossen).
            			
            			\begin{proof}
            				For a given, large $r_0$ we let $D_{r_0}$ be the bounded domain whose boundary is given by the circle $r=r_0$ in our usual coordinate notation. The Gauss-Bonnet theorem, applied to $D_{r_0}$ reads
            				\[
            				\int_{D_{r_0}}K_g +\int_{\partial D_{r_0}}\kappa_g =2\pi \chi(D_{r_0})
            				\]
            				where $\chi(D_{r_0})$ stands for the Euler characteristic of the domain in question. Now, it is straightforward to check that our decay assumptions on the metric $g$ imply 
            				\[
            				\int_{\partial D_{r_0}}\kappa_g =2\pi \sin(\alpha) (1+o(1))
            				\]
            				which is strictly positive for $r_0$ large enough (since by definition $\alpha\in (0,\pi/2]$). Hence, due to the fact that of course $\int_{D_{r_0}}K_g\geq 0$ we deduce that $\chi(D_{r_0})=1$ for all sufficiently large $r_0$ and thus $S$ is diffeomorphic to a plane. Concerning the second assertion: if $m=0$ then letting $r_{0}\to +\infty$ in the equation above implies that for any given (large) $r_0$
            				\[
            				\int_{D_{r_0}}K_g=0
            				\]
            				and thus the assumption $K_g\geq 0$ forces the Gauss curvature of $(S,g)$ to vanish at every point. The conclusion follows at once.
            			\end{proof}
            			
            			The theorem above ensures that, whenever assuming non-negative scalar curvature, the (a priori restrictive) assumption $\alpha<\pi/2$ only rules out Euclidean $\mathbb{R}^2$, in which case the conclusion of our main theorem is trivial. Also, notice that for $\alpha\neq \pi/2$ the conclusions of Lemma \ref{lem:limgeod} apply, which will turn to be extremely relevant for the arguments we are about to present. 
            			
            			\section{Min-max embedded geodesic segments}\label{sec:finscale}
            			
            			As described in the introduction, we shall present here a general existence theorem for min-max embedded geodesic segments. To state our results, we need to introduce some notation. \\
            			
            			\indent Throughout this section, we let $(N,g)$ be a complete Riemannian manifold of dimension two, without boundary. Given two points $p, q$ with $p\neq q$ we assume the existence of two embedded geodesics connecting them: let us denote by $\gamma_1:[0,1]\to N$ (resp. $\gamma_2:[0,1]\to N$) a parametrization of the first (resp. the second) of them by a constant multiple of the corresponding arc-length paramater. 
            			We further assume that the closed domain $\overline{\Omega}$ bounded by $\Gamma_1:=\textrm{spt}(\gamma_1)$ and $\Gamma_2:=\textrm{spt}(\gamma_2)$ is $C^1$-diffeomorphic to the upper half-disk in $\mathbb{R}^2\simeq\mathbb{C}$: namely there is a map $\Phi: \mathbb{D}_{+}\to \overline{\Omega}$ which is a proper diffeomorphism of class $C^1$ (the regularity of the map being understood in the sense of restriction of a $C^1$ map on open neighborhoods of $\mathbb{D}_{+}$ and $\overline{\Omega}$) for $\mathbb{D}_{+}=\left\{z\in\mathbb{C} \, : \ |z|\leq 1, \textrm{Im}(z)\geq 0 \right\}$.
            			Let then
            			\[
            			X:=\left\{\gamma\in  W^{1,2}([0,1], N) \, : \ \gamma(0)=p, \gamma(1)=q  \right\}
            			\]
            			and
            			\[
            			\Sigma:= \left\{ H\in C([0,1], X) \, : \  H(0)=\gamma_1, H(1)=\gamma_2 \right\}.
            			\]
            			The previous assumption concerning the region $\Omega$ ensures that the class $\Sigma$ is not empty.
            			
            			Thus, we shall introduce the min-max value
            			\[
            			\Lambda:= \inf_{H\in\Sigma}\max_{s\in [0,1]} E(H(s))
            			\]
            			where for an element $\gamma\in X$ 
            			\[
            			E(\gamma)=\int_{0}^{1}g(\dot{\gamma}(t),\dot{\gamma}(t))\,dt
            			\]
            			is the standard energy functional on curves (see Appendix \ref{sec:geodcurr} for further details and a recollection of some basic facts). In the setting above, we let $\textrm{Crit}(E)\subset X$ denote the set of critical points for $E$ (geodesics parametrized by a constant multiple of the arc-length). Throughout this section we set $I=[0,1]$.

            			\begin{prop}\label{prop:finscale}
            				Let $(N,g)$ be a complete Riemannian manifold of dimension two, without boundary, and for given distinct points $p, q$ assume that there exist two parametrized embedded geodesics $\gamma_1, \gamma_2: I\to N$ bounding a half-disk-type region (in the sense explained above), satisfying $\gamma_1(0)=\gamma_2(0)=p, \gamma_1(1)=\gamma_2(1)=q$, and such that the mountain-pass condition
            				\[
            				\Lambda> \max\left\{E(\gamma_1), E(\gamma_2)\right\}
            				\]
            				holds. Then there exists a parametrized embedded geodesic $\gamma_3:I\to N$, whose endpoints are $p, q$ and whose energy equals the value $\Lambda$. Furthermore, $\gamma_3$ has Morse index less or equal than one (as a critical point of the energy functional).
            			\end{prop}		
            			
            			This result would be standard if we removed the word \textsl{embedded} from the conclusion of our statement. Instead, the requirement that the third geodesic segment that we produce has no self-intersections imposes some non-trivial work, for which we shall mostly rely on the methods recently introduced in \cite{CL14}. Proposition \ref{prop:finscale} will in fact easily follow given the two ancillary lemmata that we are about to state.
            			
            			Following standard terminology in min-max theory (see e. g. \cite{CDL03}) we shall remind the reader that a sequence $\left\{H_n\right\}\in \textrm{Seq}(\Sigma)$ is called \textsl{minimizing} if 
            			\[
            			\sup_{s\in [0,1]} E(H_n(s)) \to \Lambda, \ \textrm{as} \ n\to\infty
            			\]
            			and that, in such case, a sequence $\left\{\gamma_n\right\}\in \textrm{Seq}(X)$ for $\gamma_n:=H_n(s_n)$ is called \textsl{min-max} if
            			\[
            			E(\gamma_n)\to \Lambda, \ \textrm{as} \ n\to\infty.
            			\]
            			
            			The first result is fairly basic and ensures that given a minimizing sequence one can always extract an associated min-max sequence converging (in $X$ and hence smoothly) to a stationary point (in other words: that there exists a stationary element among its limit points).
            			
            			\begin{lem}\label{lem:critlev}
            				(Setting as above). Given a minimizing sequence $\left\{H_n\right\}\in \textrm{Seq}(\Sigma)$, there exist an associated min-max sequence $\left\{H_n(s_n)\right\}$ and $\gamma_{\infty}\in \textrm{Crit(E)}\subset X$ such that $H_n(s_n)\to \gamma_{\infty}$ in $X$.
            			\end{lem}
            			
            			Of course, we remark that the claim that \textsl{every} min-max sequence should converge to an element in $Crit(E)$ is patently false, as is discussed in \cite{Pit81} and \cite{CDL03}. To overcome such issue, one needs to perform a \textsl{pull-tight} procedure, which is then needed when discussing the regularity of min-max minimal surfaces. That step is not really necessary here.\\
            			
            			\indent	Roughly speaking, we can identify the tangent space of $X$ at $\gamma$ with the space of $W^{1,2}$ sections of the tangent bundle of $N$ restricted to (the support of) $\gamma$ and vanishing at the endpoints:
            			\begin{multline*}
            			T_{\gamma}X  \\
            			=\left\{V\! \in W^{1,2}([0,1], TN) :  \,
            			\pi(V(t))=\gamma(t) \ \forall  t\in [0,1], 
            			 V(0)=V(1)=0
            			  \right\}
            			\end{multline*}
            			where $\pi:TN\to N$ is the standard projection of the tangent bundle onto its base manifold. Notice also that we can then naturally endow $X$ with the structure of a Riemannian Hilbert manifold $(X,g^{X})$  by simply setting
            			\[
            			g^{X}: T_{\gamma}X\times T_{\gamma}X\to\mathbb{R}, \ \ g^{X}(V_1, V_2)=\int_0^{1}\left(g(V_1, V_2)+g(\dot{V}_1,\dot{V}_2)\right)\,dt.
            			\]
            			We refer the reader to Chapter 1 of \cite{Kli78} for an ample discussion and contextualization of these notions.\footnote{While \cite{Kli78} is mostly focused on the case of \textsl{closed} geodesics, modifying the basic definitions and constructions to deal with the case of curves with fixed endpoints only requires minimal effort.}We can now proceed with the proof of Lemma \ref{lem:critlev}.
            			
            			\begin{proof}
            				Let us start by describing the basic idea behind the argument we are about to present. Arguing by contradiction, we shall see that if a minimizing sequence did not have the property above (namely: if all converging min-max sequences clustered to non-stationary points) then one could indeed perform an unobstructed deformation of the minimizing sequence $\left\{H_n\right\}$ thereby obtaining a new sequence $\left\{\overline{H}_n\right\} \in \textrm{Seq}(\Sigma)$ for which
            				\[
            				\limsup_{n\to\infty} \sup_{s\in [0,1]}E(\overline{H}_n(s))<\Lambda
            				\]	
            				that is impossible, by the very definition of $\Lambda$ as min-max value. In the one-dimensional setting we are dealing with, the aforementioned deformation is performed by means of the so-called $H^1$-gradient flow (in fact: steepest descent flow) for the functional $E$ on $X$. 
            				
            				\
            				
            				Let us preliminarily observe that, under the contradiction assumption above, we can assume (without loss of generality) that for \textsl{any} min-max sequence $\left\{H_n(s_n)\right\} \in \textrm{Seq}(X)$ 
            				\[
            				\liminf_{n\to\infty} ||\nabla E(H_n(s_n))||_{X}=\delta>0
            				\] 
            				for, if not, the fact that the energy functional satisfies the Palais-Smale condition would imply sub-convergence of such sequence to a stationary critical point of $E$.
            				That being said, fix any $\tau^{\ast}>0$ and consider for each $\gamma\in X$ the $\tau^{\ast}$-image $\tilde{\gamma}$ of $\gamma$ under the gradient flow of $E$ (with respect to the Riemannian structure defined above), namely
            				we set $\tilde{\gamma}=\Phi_{\tau^{\ast}}(\gamma)$ where $\Phi_{\tau}$ is the flow associated to the ODE
            				\[
            				\begin{cases}
            				\frac{d}{d\tau}\gamma_{\tau}=-\nabla E(\gamma_{\tau}) \\
            				\gamma(0)=\gamma
            				\end{cases}
            				\]
            				and $\nabla E$ is defined by the equation $g^{X}(\nabla E(\gamma), V)=dE(\gamma)[V]$ to hold for all $V\in T_{\gamma}X$ (it is easily checked that $dE(\gamma)[V]=2\int_{0}^1 g(\dot{\gamma},\dot{V})\,dt$).
            				Recall that
            				\[
            				\frac{d}{d\tau}E(\gamma_{\tau})=-\|\nabla E(\gamma_{\tau})\|_X^2,
            				\]
            				so that for any $\tau'<\tau''$
            				\[
            				E(\gamma_{\tau''})-E(\gamma_{\tau'})=-\int_{\tau'}^{\tau''}\|\nabla E(\gamma_s)\|_X^2\,ds
            				\]
            				which we shall repeatedly use in the sequel of this proof.
            				In particular, it is convenient for any $n\geq 1$ to set  $\tilde{H}_n(s)=\Phi_{\tau^{\ast}}(H_n(s))$ and thus consider $\small\{\tilde{H}_n\small\}\in \textrm{Seq}(\Sigma)$ which is still (patently) a minimizing sequence\footnote{Let us remark that the fact that $\gamma_1, \gamma_2$ are geodesics is implicitly used in this step, as it implies that the class $\Sigma$ is indeed stable under the flow $\Phi$ and in particular $\small\{\tilde{H}_n \small\}\in Seq(\Sigma)$ since it was assumed that $\left\{H_n \right\}\in Seq(\Sigma)$.}due to the monotonicity of the flow in question.  Hence, for any $n\geq 1$ pick $\tilde{s}_n\in \arg \max E(\tilde{H}_n(s))$ and consider the associated sequence of (pre-flow) curves $\left\{H_n(\tilde{s}_n)\right\}\in \textrm{Seq}(X)$. It follows that it must be
            				\[
            				\liminf_{n\to\infty}E(H_n(\tilde{s}_n))=\Lambda
            				\]
            				for otherwise
            				\[
            				\Lambda=\liminf_{n\to\infty}E(\tilde{H}_n(\tilde{s}_n))\leq \liminf_{n\to\infty}E(H_n(\tilde{s}_n))<\Lambda
            				\]
            				which is impossible. This is to say that $\left\{H_n(\tilde{s}_n)\right\}\in \textrm{Seq}(X)$ is itself a min-max sequence. Also, the preliminary remark specifies to such sequence to ensure that indeed $\liminf_{n\to\infty} ||\nabla E(H_n(\tilde{s}_n))||_{X}=\delta$ for some number $\delta>0$. Pick then an intermediate threshold $\delta'\in (0,\delta)$ and set
            				\[
            				\tau_n:=\inf\left\{\tau\in (0,\tau^{\ast}] \, : \  ||\nabla E(\Phi_{\tau}(H_n(\tilde{s}_n)))||_{X}<\delta' \right\}
            				\]
            				where we agree to define $\tau_n=\tau^*$ if the set in question is empty.
            				
            				The following dichotomy holds: \ul{either} we can extract a subsequence of indices $\left\{n_{k}\right\}$ such that $\left\{\tau_{n_k}\right\}$ has a positive bound $\tau_{\ast}\in (0,\tau^{\ast})$ (in which case it is immediately checked that  $\liminf_{k\to\infty}E(\tilde{H}_{n_k}(\tilde{s}_{n_k}))\leq \Lambda-\delta'\tau_{\ast}$, which is impossible, as we have already observed) \ul{or} instead such condition is violated for all $\delta'>0$ and thus we can find
            				\begin{enumerate}
            					\item{a sequence $\left\{\delta_k\right\}$ with $\delta_k \searrow 0$;}
            					\item{a sequence $\left\{n_k\right\}$ with $n_k\nearrow \infty$;}
            					\item{a sequence $\left\{\tau_{n_k}\right\}$ with $\tau_{n_k}\searrow 0$;}
            				\end{enumerate}
            				such that
            				\[
            				\|\nabla E(\Phi_{\tau_{n_k}}(H_{n_k}(\tilde{s}_{n_k})))||_{X}<\delta_k.
            				\]
            				But, if this were the case, again by Palais-Smale the min-max sequence $\left\{\Phi_{\tau_{n_k}}(H_{n_k}(\tilde{s}_{n_k}))\right\}$ would subconverge to a stationary point of $E$ on $X$, say $\tilde{\gamma}_{\infty}$. Recall next the following elementary estimate for 
            				solutions $\gamma_s$ of the gradient flow of $E$:
            				\begin{multline*}
            				d_X(\gamma_{\tau},\gamma_{\tau_0})\leq \int_{\tau_0}^\tau \|{\textstyle{\frac{d}{ds}}} \gamma_s\|_X\, ds = \int_{\tau_0}^{\tau}\|\nabla E(\gamma_s)\|_X\,ds \\
            				\leq \left(\int_{\tau_0}^{\tau}\|\nabla E(\gamma_s)\|_X^2\,ds\right)^{1/2}|\tau-\tau_0|^{1/2}\, 
            				\end{multline*}
            				where $d_X(\cdot,\cdot):X\times X\to\mathbb{R}$ denotes the (Riemannian) path-distance in the Hilbert manifold $X$.
            				We apply it with $\tau_0=0$ and $\tau= \tau_{n_k}$ to bound the distance between $H_{n_k} (\tilde{s}_{n_k})$ and $\Phi_{\tau_{n_k}}(H_{n_k}(\tilde{s}_{n_k}))$. Since 
            				\begin{align*}
            				\int_{0}^{\tau_{n_k}}\|\nabla E\left(\Phi_{\tau}(H_{n_k}(\tilde{s}_{n_k}))\right) \|_X^2\,d\tau &= E(H_{n_k}(\tilde{s}_{n_k}))-E(
            				\Phi_{\tau_{n_k}} (H_{n_k} (\tilde{s}_{n_k}))) \\
            				&\leq E (H_{n_k}(\tilde{s}_{n_k}))
            				\end{align*}
            				is uniformly bounded in $k$ and $\tau_{n_k}\downarrow 0$, we conclude that $\left\{H_{n_k}(\tilde{s}_{n_k})\right\}$ and $\left\{\Phi_{\tau_{n_k}}(H_{n_k}(\tilde{s}_{n_k}))\right\}$ have the same limit, namely the stationary point $\tilde{\gamma}_{\infty}$. Once again, this contradicts our initial assumption, namely that all converging min-max sequences cluster to non-stationary points, and thereby our proof is complete.
            			\end{proof}	 
            			
            			Now, before describing the resolution of singularities procedure we need to reduce to \textsl{generic} homotopies, in the sense made precise by this statement.
            			
            			\begin{lem}\label{lem:generic}
            			(Setting as above).	Given $\varepsilon>0$ the following holds: for every $H\in \Sigma$ there exists $\tilde{H}\in \Sigma$, in fact $\tilde{H}\in C^{\infty}([0,1]\times [0,1]; N)$, such that all of these assertions are true:
            				\begin{enumerate}
            					\item{$\forall \ s\in [0,1]$ one has $\|H(s)-\tilde{H}(s)\|_{X}<\varepsilon$;}	
            					\item{there are finitely many singular times $s_1<s_2<\ldots<s_{k-1}<s_{k}$ and if $s\in [0,1]\setminus \cup_{i=1}^{k}\left\{s_i\right\}$ the map $\tilde{H}(s)$ is an immersion with only transverse self-intersections and no triple points;}
            					\item{for $s=s_i$ the singular events\footnote{The reader is referred to Section 2 of \cite{CL14} for relevant definitions, see in particular Figure 1 therein for a clear illustration of the three possible Reidemeister moves.}correspond to one of the standard three Reidemeister moves and, furthermore, singular events do not happen concurrently;}
            					\item{there exist $\delta>0$ such that $\forall s\in [0,1]$ the curve $\tilde{H}(s)$ has no self-intersections in $B_{\delta}(p)\sqcup B_{\delta}(q)$, and is an immersion when restricted to these balls.}	
            				\end{enumerate}		
            			\end{lem}		
            			
            			\begin{proof}
            				First of all, let us see why (4) is indeed a generic condition: in other words, given $\varepsilon>0$ and $H\in\Sigma$ as per the statement above, let us show that we can find $\tilde{H}\in \Sigma$ so that both (1) and (4) are satisfied. This is intuitively clear, but let us discuss it for the sake of completeness. We shall refer to the first claim (no self-intersections) and leave the second one to the reader.
            				For our fixed $p\in N$ (the argument will then be applied to the point $q$ as well) consider
            				\[
            				Q^{(2)}=\left\{ (t_1, t_2)\in [0,1]\times [0,1] \, : \ t_1\neq t_2 \right\}
            				\]
            				as well as the twofold 0-jet bundle
            				\[
            				J_2^0([0,1], N)= Q^{(2)}\times N^2
            				\]
            				and its subset $N_d=\left\{(t_1, t_2, p, p) \right\}\subset J_2^0([0,1], N)$. Let us observe that $N_d$ is a 2-dimensional closed submanifold in $J_2^0$ which, in turn, is a smooth (open) manifold of dimension 6.
            				Given $s\in [0,1]$ one has the induced map $j_2^0 H(s):Q^{(2)}\to J_2^0([0,1], N)$ defined by the equation
            				\[
            				j_2^0 H(s)(t_1, t_2)=(t_1, t_2, H(s)(t_1), H(s)(t_2))
            				\]
            				so that clearly 
            				\[
            				H(s)(t_1)=H(s)(t_2)=p \ \iff \ j_2^0 H(s)(t_1, t_2)\in N_d.
            				\]
            				Hence, one can observe that for a generic homotopy $H$ the map 
            				\[f:[0,1]\times Q^{(2)}\to J_2^0([0,1], N), \ f(s,t_1, t_2)=j_2^0 H(s)(t_1, t_2)
            				\] will intersect the submanifold $N_d$ transversely which implies (by dimensional counting) that the intersection $f\pitchfork N_d$ in question will in fact be empty. Thus, we can find $\tilde{H}\in \Sigma$ which is $\varepsilon$-close to $H$ and in a way that self-intersections do not happen at $p$ or $q$ so that, by compactness we can indeed find $\delta>0$ (depending on $H$ and $\varepsilon$) so that all self-intersections of $\tilde{H}(s)$ as $s$ varies in $[0,1]$ happen outside of the balls of center $p$ (resp. $q$) and radius $3\delta$. Similarly, one proves that the map $\tilde{H}$ can be chosen, generically, so that $\tilde{H}(s)$ is  an immersion near $p$ and $q$ for all $s\in [0,1]$. At that stage, one can follow verbatim the (analogous) transversality arguments presented in \cite{CL14}, pg. 1083-1084 to ensure that by further perturbing $\tilde{H}$  the other conditions (2) and (3) can be accomodated as well.
            			\end{proof}
            			
            			At this stage, we are ready to use the machinery of \cite{CL14} to our scopes.

            			\begin{lem}\label{lem:surg}
            				(Setting as above).
            				Given $\varepsilon>0$ the following holds: for every $\tilde{H}\in\Sigma$ \textsl{generic} homotopy (in the sense specified by Lemma \ref{lem:generic}) there exists $\overline{H}\in\Sigma$, in fact $\overline{H}\in C^{\infty}([0,1]\times [0,1], N)$, such that these assertions are true:
            				\begin{enumerate}
            					\item{$\forall \ s \in [0,1]$ one has $E(\overline{H}(s))\leq E(\tilde{H}(s))+\varepsilon$;}
            					\item{the map $\overline{H}$ is an isotopy, namely $\forall \  s\in [0,1]$ the map $\overline{H}(s)$ is an embedding.}	
            				\end{enumerate}		
            			\end{lem}
            			
            			\begin{proof} Given condition (4) of Lemma \ref{lem:generic}, we know that for a generic homotopy all self-intersections (and singularities) happen away from $B_{\delta}(p)\sqcup B_{\delta}(q)$. Hence, one can perform a finite chain of Reidemeister moves according to the general algorithm described in \cite{CL14}, pg. 1088-89, with the only constraint that (while resolving the singularities) no modifications should be made to the support of the homotopy inside  $B_{\delta/2}(p)\sqcup B_{\delta/2}(q)$.
            			\end{proof}	
            			
            			We shall then proceed with the proof of Proposition \ref{prop:finscale}.
            			
            			\begin{proof}
            				Let $\left\{H_n\right\}\in \textrm{Seq}(\Sigma)$ be a minimizing sequence for the min-max problem defined above. By applying, one after the other, Lemma \ref{lem:generic} and Lemma \ref{lem:surg} (taking in both cases $\varepsilon=1/2n$ when dealing with $H_n$) we can produce a new sequence $\left\{\overline{H}_n\right\}\in \textrm{Seq}(\Sigma)$ which is also minimizing (for indeed $E(\overline{H}_n(s))\leq E(H_n(s))+1/n$) and consists of isotopies. Now, Lemma \ref{lem:critlev} ensures the existence of an associated min-max sequence $\left\{\overline{H}_n(s_n)\right\}\in \textrm{Seq}(X)$ converging in $X$ to a geodesic $\gamma_{\infty}:I\to N$ attaining energy $\Lambda$.
            				Of course, $\gamma_{\infty}$ may a priori not be an embedding, but if it had self-intersections those would be transverse (by virtue of a standard ODE uniqueness argument). However, since in particular the sequence $H_{n}(s_n)$ does converge to $\gamma_{\infty}$ uniformly (namely: in $C^0$) if $\gamma_{\infty}$ had a transverse self-intersection, then $H_n(s_n)$ would not be an embedding for $n$ large enough, and this contradiction completes the proof for what concerns the existence statement. 
            				Lastly, the fact that the geodesic $\gamma_3\in X$, regarded as a critical point of the energy functional $E(\cdot)$, has Morse index less or equal than one is a general fact about one-dimensional mountain-pass schemes. In particular, denote by $A_\infty$ the subset of $X$ consisting of the cluster points of any min-max sequence $\{\overline{H}_n (s_n)\}$ and let $K_{\Lambda}$ be the set of critical points $\gamma$ of $E$ with $E (\gamma) = \Lambda$. Observe that the argument above implies that any element of $A_\infty \cap K_\Lambda$ is embedded. We can then apply Theorem 4 of page 53 in \cite{Gho91} to conclude the existence of at least one element in $A_\infty \cap K_\Lambda$ whose Morse index is at most $1$ (with reference to the notation of \cite{Gho91}, note that such theorem can be applied because the group $G$ is in our case the trivial group and the $G$-invariant set $F$ is the whole space $X$; in particular $K_\Lambda \cap F = K_\Lambda$ is trivially an ``isolated critical set in itself'' in the sense of \cite{Gho91}). 
            			\end{proof}

            			\section{The construction of embedded geodesic lines}\label{sec:proof}

            			Given the above preliminaries, we shall prove here a multiplicity theorem for geodesics with fixed endpoints: given two antipodal points $p, q$ on an asymptotically conical surface we want to show the existence of (at least) three geometrically distinct embedded geodesics connecting them. Let us start by reminding the reader that, by virtue of Lemma \ref{lem:limgeod} we already know the existence of one such geodesic (namely: the absolute length-minimizer) and the next step we are about to present is the construction of a second one by means of a localized minimization argument.
            			
            			\begin{lem}\label{lem:second}
            				Let $(S,g)$ be an asymptotically conical surface of non-negative scalar curvature and (in asymptotic coordinates $(r,\varphi)$) let $p, q$ be a couple of antipodal points. Let us denote the value of their first coordinate by $r_0$. Then there exists a constant $\overline{r}_0$ (only depending on $(S,g)$) such that for every $r_0\geq \overline{r}_0$ there are two distinct, simple geodesics $\Gamma_1, \Gamma_2$ connecting $p, q$, they are disjoint and geometrically converge, under rescaling by a factor $r^{-1}_0$ respectively to $\Gamma', \Gamma''$ (modulo renaming of the latter ones). 
            			\end{lem}	
            			
            			\begin{proof}
            				Let us denote by $\Gamma_1$ the length-minimizing (hence necessarily simple) geodesic between $p$ and $q$, whose existence is guaranteed by Lemma \ref{lem:limgeod}. Possibly by renaming, the same result implies that $\Gamma_1$ will converge, upon rescaling, to $\Gamma'$.
            				In order to construct the second geodesic, it is convenient to identify the domain covered by our $(r,\varphi)$ charts with (a suitable subset of) a planar wedge of angle $2\pi\sin\alpha$, namely with the wedge described in polar coordinates $(\rho,\vartheta)$ by the equations $0<\rho, 0\leq\vartheta\leq 2\pi\sin\alpha$, with pointwise identification of the two edges at $\vartheta=0, \vartheta=2\pi\sin\alpha$ (see also Remark \ref{ref:wedgecoor2}). In this model, we can then assume that the points $p, q$ have coordinates given by
            				\[
            				\rho(p)=\rho(q)=r_0, \ \ \vartheta(p)=\frac{\pi}{2}\sin\alpha, \vartheta(q)=\frac{3\pi}{2}\sin\alpha
            				\]  
            				and in turn we let $\tilde{\Gamma}_2$ be identified with the straight segment connecting them (assuming that $\Gamma''$ is the straight segment gotten from $\tilde{\Gamma}_2$ by rescaling via a factor $r_0$). For $\delta>0$ small, to be fixed later, let us consider the ellipse centered at the midpoint $m$ of $p, q$ (that is to say $\rho(m)=r_0\cos \left(\frac{\pi}{2}\sin\alpha\right), \vartheta(m)=\pi\sin\alpha$) and axes of length  $(1+\delta)r_0\sin\left(\frac{\pi}{2}\sin\alpha\right)$ (parallel to $\tilde{\Gamma}_2$) and $(1-\delta)r_0\cos\left(\frac{\pi}{2}\sin\alpha\right)$ (orthogonal to $\tilde{\Gamma}_2$ ): such ellipse has positive geodesic curvature in the flat metric the wedge is endowed with and, since this model is isometric to the (limit) cone $C_{\alpha}$, we can just pick $D\subset S$ to be the image of the interior of this ellipse under the identification map above. The decay assumptions on $e$ (the \textsl{error terms} of the metric $g$) ensure that this domain will indeed be mean-convex for $r\geq \overline{r}_0$ sufficiently large and of course we can pick $\delta>0$ small enough that $D$ does not cover the whole $\Gamma_1$, but just small neighborhoods of the common vertices of $\Gamma_1$ and $\tilde{\Gamma}_2$, cf. Figure \ref{fig:ellisse}.
            				
            				\begin{figure}[htbp]
            					\begin{center}
            						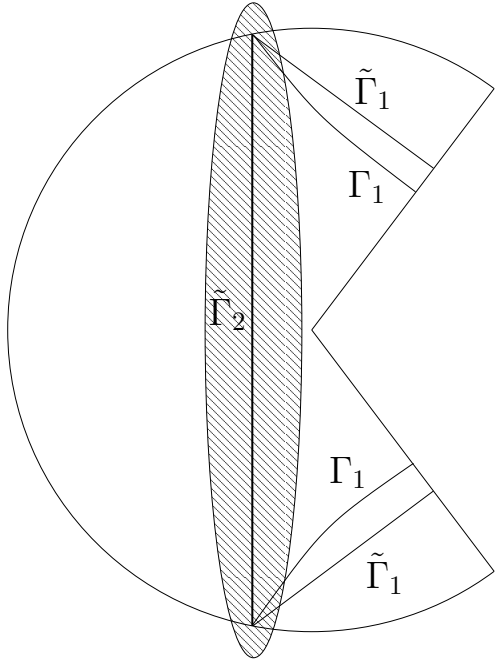
            					\end{center}
            					\caption{The shadowed region is $D$ in the coordinates $(\rho, \vartheta)$. The lines $\tilde{\Gamma}_1$ and $\tilde{\Gamma}_2$ rescaled by a factor $r_0$ coincide with $\Gamma'$ and $\Gamma''$, respectively.}\label{fig:ellisse}
            				\end{figure}
            				
            				That being said, standard direct methods ensure the existence of a smooth geodesic $\Gamma_2$ connecting $p$ to $q$ and having shortest length among those contained in $D$. Being locally length-minimizing, $\Gamma_2$ cannot have self-intersections which means it has to be a simple geodesic.
            				By construction (specifically: by the choice of $\delta$) the rescalings of those $\Gamma_2$ (as we let $r_0\to+\infty$ and rescale by $r_0$) must locally converge to a geodesic of $C_{\alpha}$ which cannot be either $\Gamma'$ or $\Gamma'''$ and hence must be $\Gamma''$, as we had claimed. Lastly: $\Gamma_1$ and $\Gamma_2$ cannot intersect for, if they did, we could shorten the length of either of them by means of local cut-and-paste operations, thereby violating their minimizing properties. This completes the proof. 
            			\end{proof}

            			We can now employ the result obtained in Section \ref{sec:finscale} to produce a third geodesic segment connecting any two fixed antipodal points $p, q$ on $(S,g)$. From now onwards, due to the limit arguments we will have to perform, it is useful to make the dependence on $(r_0,\varphi_0)$ explicit for all objects we deal with: in particular we shall write $p^{(r_0,\varphi_0)},q^{(r_0,\varphi_0)}$ and denote by $\Gamma_1^{(r_0,\varphi_0)},\ \Gamma_2^{(r_0,\varphi_0)}$ the (supports of the) geodesic segments constructed above. 
            			Let us further denote by $\gamma_1^{(r_0,\varphi_0)}:[0,1]\to S$ (resp. $\gamma_2^{(r_0,\varphi_0)}:[0,1]\to S$) a parametrization of $\Gamma_1^{(r_0,\varphi_0)}$ (resp. $\Gamma_2^{(r_0,\varphi_0)}$) by a constant multiple of the corresponding arc-length paramater. 
            			Let then the spaces $X (r_0,\varphi_0), \Sigma(r_0,\varphi_0)$ and the min-max value $\Lambda(r_0,\varphi_0)$ be defined as above.
            			
            			A crucial remark is that, due to the fact that $S$ is diffeomorphic to $\mathbb{R}^2$ (by virtue of Theorem \ref{lem:pmt2}) the class $\Sigma(r_0,\varphi_0)$ is not empty.

            			\begin{prop}\label{pro:minmaxgeod} In the setting described above, we have
            				\[
            				\liminf_{r_0\to+\infty} \frac{\Lambda(r_0,\varphi_0)}{r^2_0}\geq 4.
            				\]
            				
            				As a result, for every $r_0\geq \overline{r}_0$ there exists a third simple geodesic $\Gamma_3^{(r_0,\varphi_0)}$ connecting $p^{(r_0,\varphi_0)}$ to $q^{(r_0,\varphi_0)}$ whose parametrization $\gamma_3: [0,1]\to S$ has constant speed (a constant multiple of the arc-length parameter) and attains the min-max value $\Lambda(r_0,\varphi_0)$, namely $\gamma_3^{(r_0,\varphi_0)}\in X(r_0,\varphi_0)$ and
            				\[
            				E(\gamma_3^{(r_0,\varphi_0)})=\Lambda(r_0,\varphi_0).
            				\]
            				Furthermore $\gamma_3^{(r_0,\varphi_0)}$ has Morse index less or equal than one for the functional $E$.
            			\end{prop}	
            			
            			\begin{proof} We need to start by checking that the mountain-pass condition holds, at least when the antipodal points serving as boundary are far away in the asymptotic region.\\
            				
            				\indent Given $\varepsilon>0$, it follows from Lemma \ref{lem:second} that one can find a possibly larger $\overline{r}_0$ such that
            				\[
            				\frac{E\left(\gamma_i^{(r_0,\varphi_0)}\right)}{r^2_0}\leq 4(1+\varepsilon)\sin^2\left(\frac{\pi}{2}\sin\alpha \right), \ \ i=1,2.
            				\]
            				for every $r_0\geq \overline{r}_0$.
            				On the other hand, we claim that necessarily
            				\[
            				\frac{\Lambda(r_0,\varphi_0)}{r^2_0}\geq 4(1-\varepsilon).
            				\]
            				To see this, let us start by observing that due to the natural embedding $W^{1,2}([0,1],S)\hookrightarrow C([0,1], S)$ the set $\Sigma(r_0,\varphi_0)$ is included in the class of (continuous) homotopies connecting $\gamma_1^{(r_0,\varphi_0)}$ with $\gamma_2^{(r_0,\varphi_0)}$: given $H\in\Sigma$ we can simply set $\tilde{H}:[0,1]\times [0,1]\to S, \ \tilde{H}(s,t)=H(s)(t)$. Thus, considered for any fixed large $r_0$ the region $\Omega=\Omega(r_0,\varphi_0)$ bounded by $\Gamma_1$ and $\Gamma_2$ (which, let us recall, is homeomorphic to a topological disk), and fixed $o\in \cap _{r_0\geq \overline{r}_0}\Omega(r_0,\varphi_0)$ for any $H\in\Sigma$ one can find a couple $(s_0,t_0)\in (0,1)\times (0,1)$ such that $H(s_0,t_0)=o$. Said $\gamma_{s_0}^{(r_0,\varphi_0)}:[0,1]\to S$ the corresponding path (namely: $\gamma_{s_0}^{(r_0,\varphi_0)}=H(s_0)$) we claim that
            				\[
            				\frac{E(\gamma_{s_0}^{(r_0,\varphi_0)})}{r^2_0}\geq 4(1-\varepsilon)
            				\] 
            				which would of course imply the claim given the fact that $H$ is an arbitrary element in the class $\Sigma$. This is shown as follows: said $\tilde{\gamma}^{(r_0,\varphi_0)}:[0,1]\to S$ the broken geodesic gotten by taking a length-minimizing curve connecting $p$ to $o$ concatenated to a length-minimizing curve connecting $o$ to $q$, the concatenation being smoothened near the junction point, we have (by Cauchy-Schwarz)
            				\[
            				\frac{E(\gamma_{s_0}^{(r_0,\varphi_0)})}{r^2_0}\geq \frac{L^2(\gamma_{s_0}^{(r_0,\varphi_0)})}{r^2_0}\geq \frac{L^2(\tilde{\gamma}^{(r_0,\varphi_0)})}{r^2_0}\geq 4(1-\varepsilon)
            				\] 
            				where the last inequality follows from the observation that as we rescale by a factor $r_0$ the support $\tilde{\Gamma}$ (which must locally converge to a geodesic) can only converge to $\Gamma'''$ by Lemma \ref{lem:clairaut} together with the fact that the choice of $o$ is independent of $r_0$.
            				That being gained, let us choose once and for all $\varepsilon>0$ (depending only on $(S,g)$, in fact just on the asymptotic opening angle $\alpha$) such that
            				\[
            				\frac{1-\varepsilon}{1+\varepsilon}> \sin^2\left(\frac{\pi}{2}\sin\alpha \right).
            				\]
            				This ensures that the mountain-pass gap condition required by Proposition \ref{prop:finscale} is satisfied and so we immediately derive the existence of a critical point $\gamma_3^{(r_0,\varphi_0)}\in X$ attaining the min-max value $\Lambda(r_0,\varphi_0)$. Lastly, we notice that $\gamma_3^{(r_0,\varphi_0)}$ being a critical-point of $E(\cdot)$ is also a critical point of $L(\cdot)$, parametrized by constant speed. It follows from our argument above that
            				\[
            				\liminf_{r_0\to+\infty} \frac{\Lambda(r_0,\varphi_0)}{r^2_0}\geq 4(1-\varepsilon)
            				\]
            				but of course this applies to any $\varepsilon$ as small as we wish, so the last assertion follows at once.
            			\end{proof}	
            			
            			In order to rule out \textsl{concentration phenomena} when taking limits of min-max geodesic segments (as the endpoints drift off to infinity), we need the following lemma, which concerns the explicit description of effective sweepouts and thereby provides an upper bound on the min-max value.
            			
            			\begin{lem}\label{lem:upperb} (Setting as above.) For every $\varphi_0\in S^1$ one has that the min-max values satisfy
            				\[
            				\limsup_{r_0\to+\infty}\frac{\Lambda(r_0,\varphi_0)}{r^2_0}\leq 4.
            				\]
            			\end{lem}

            			\begin{rmk}
            				For the following proof, it turns out to be more convenient to work with the coordinates $\left\{\rho,\vartheta\right\}$	defined in Remark \ref{ref:wedgecoor2} and to treat $(S,g)$ (minus a compact set) as an Euclidean wedge of angle $2\pi\sin\alpha$ with pointwise identification of the two sides. 
            			\end{rmk}	
            			
            			\begin{proof}
            				Given $\varepsilon>0$ fix (once and for all) a large scale $\rho_{-}$ such that the metrics $g$ and $\delta$ differ on the domain of $(S,g)$ defined in coordinates by $\rho\geq \rho_{-}/2$ for an amount less than $\varepsilon^2$ in $C^{2}$ norm. Consider the two linear segments joining $p_{-}$ to $q_{-}$ (where $\rho(p_{-})=\rho(q_{-})=\rho_{-}$ and $\vartheta({p_{-}})=\pi\sin\alpha$ while $\vartheta({q_{-}})\in \left\{0,2\pi\sin\alpha \right\}$ since the latter is represented by two points that are geometrically identified): such paths can be parametrized by means of the coordinate equations given by
            				\[
            				\gamma^{-}_{1}(t)=\left(\rho_{-}\frac{\sin\left\{\frac{\pi}{2}\left(1-\sin\alpha\right) \right\}}{\sin\left\{\frac{\pi}{2}\left(1+(1-2t)\sin\alpha\right)\right\}}, (1+t)\pi\sin\alpha \right)
            				\]
            				and
            				\[
            				\gamma^{-}_{2}(t)=\left(\rho_{-}\frac{\sin\left\{\frac{\pi}{2}\left(1-\sin\alpha\right) \right\}}{\sin\left\{\frac{\pi}{2}\left(1+(1-2t)\sin\alpha\right)\right\}}, (1-t)\pi\sin\alpha \right)
            				\]	
            				as can be checked by means of some elementary trigonometry.
            				Furthermore, let $H^{-}\in\Sigma$ be an homotopy connecting $\gamma^{-}_1$ to $\gamma^{-}_{2}$ (whose existence is a consequence of Theorem \ref{lem:pmt2}).
            				We shall now define an homotopy $H$ connecting $\gamma^{+}_1$ to $\gamma^{+}_2$, (parametrizations of) the two stable geodesics constructed above (see Lemma \ref{lem:second}) for endpoints $p_{+}, q_{+}$ having coordinates $\rho(p_{+})=\rho(q_{+})=\rho_{+}$ and $\vartheta({p_{+}})=\pi\sin\alpha$ while $\vartheta({q_{+}})\in \left\{0,2\pi\sin\alpha \right\}$ where $\rho_{+}$ is a \textsl{free} large scale, much larger than $\rho_{-}$ (so that we will then take $\rho_{+}\to+\infty$). To that aim, let us recall that $\Gamma_1, \Gamma_2$ converge geometrically, when rescaled as explained above, respectively to $\Gamma',\Gamma''$ so that (for large $\rho_{+}$ but uniformly in $\vartheta$) we can find constant speed parametrizations of $\Gamma_1, \Gamma_2$ (which we shall indeed call $\gamma^{+}_{1}$, $\gamma^{+}_{2}$) such that $\left\|\gamma^{+}_{i}-\check{\gamma}^{+}_{i} \right\|_X<\rho_{+}\varepsilon^2$ for $i=1,2$ where
            				\[
            				\check{\gamma}^{+}_{1}(t)=\left(\rho_{+}\frac{\sin\left\{\frac{\pi}{2}\left(1-\sin\alpha\right) \right\}}{\sin\left\{\frac{\pi}{2}\left(1+(1-2t)\sin\alpha\right)\right\}}, (1+t)\pi\sin\alpha \right)
            				\]
            				and
            				\[
            				\check{\gamma}^{+}_{2}(t)=\left(\rho_{+}\frac{\sin\left\{\frac{\pi}{2}\left(1-\sin\alpha\right) \right\}}{\sin\left\{\frac{\pi}{2}\left(1+(1-2t)\sin\alpha\right)\right\}}, (1-t)\pi\sin\alpha \right)\, .
            				\]	
            				
            				The homotopy $H$ will not belong to the class $\Sigma$ (due to corner points) but will be then be approximated by suitable smoothings while keeping length (and energy) under control.
            				Precisely, we let $H:[0,1]\times [0,1]\to S$ be defined by
            				\[ H(s,t)=
            				\begin{cases}
            				(1-8s)\gamma^{+}_{1}(t)+8s\check{\gamma}^{+}_1(t) & \textrm{for} \ 0\leq s\leq \frac{1}{8} \\
            				\\
            				\gamma^{8\left(s-1/8\right),down}_1 \ast \gamma^{8\left(s-1/8\right),tan}_1 \ast \gamma^{8\left(s-1/8\right),up}_1 (t)
            				& \textrm{for} \ \frac{1}{8}\leq s\leq \frac{1}{4} \\
            				\\
            				\gamma^{1,down}_1 \ast H^{-}\left(2\left(s-\frac{1}{4}\right),t\right)\ast\gamma^{1,up}_1 (t) & \textrm{for} \ \frac{1}{4}\leq s\leq \frac{3}{4} \\
            				\\
            				\gamma^{8\left(s-3/4\right),down}_2 \ast \gamma^{8\left(s-3/4\right),tan}_2 \ast \gamma^{8\left(s-3/4\right),up}_2 (t)
            				& \textrm{for} \ \frac{3}{4}\leq s\leq \frac{7}{8} \\
            				\\
            				8(1-s) \check{\gamma}^{-}_{2}(t)+8\left(s-\frac{7}{8}\right)\gamma^{+}_2(t) & \textrm{for} \ \frac{7}{8}\leq s\leq 1
            				\end{cases}
            				\]
            				where $\ast$ denotes the standard concatenation of paths, and we have set (once again in coordinates $\left\{\rho,\vartheta\right\}$)
            				\[
            				\gamma^{s,down}_1(t)= \left(t((1-s)\rho_{+}+s\rho_{-})+(1-t)\rho_{+}, \pi\sin\alpha\right),
            				\]
            				\[ 
            				\gamma^{s,down}_2(t)= \left(t((1-s)\rho_{+}+s\rho_{-})+(1-t)\rho_{+}, \pi\sin\alpha\right),
            				\]
            				\[
            				\gamma^{s,up}_1(t)= \left(\left(1-t\right)\left(\left(1-s\right)\rho_{+}+s\rho_{-}\right)+t\rho_{+}, 2\pi\sin\alpha\right), 
            				\]
            				\[
            				\gamma^{s,up}_2(t)= \left(\left(1-t\right)\left(\left(1-s\right)\rho_{+}+s\rho_{-}\right)+t\rho_{+},0 \right) 
            				\]
            				and
            				\[
            				\gamma^{s,tan}_{1}(t)=\left(\left((1-s)\rho_{+}+s\rho_{-} \right)\frac{\sin\left\{\frac{\pi}{2}\left(1-\sin\alpha\right) \right\}}{\sin\left\{\frac{\pi}{2}\left(1+(1-2t)\sin\alpha\right)\right\}}, (1+t)\pi\sin\alpha\right)
            				\]
            				\[
            				\gamma^{s,tan}_{2}(t)=\left(\left((1-s)\rho_{+}+s\rho_{-} \right)\frac{\sin\left\{\frac{\pi}{2}\left(1-\sin\alpha\right) \right\}}{\sin\left\{\frac{\pi}{2}\left(1+(1-2t)\sin\alpha\right)\right\}}, (1-t)\pi\sin\alpha\right).
            				\]
            				For the reader's convenience Figure \ref{fig:omotopia} gives a brief description of the five stages of the homotopy in the $(\rho, \vartheta)$ coordinates. 
            				
            				\begin{figure}[htbp]
            					\begin{center}
            						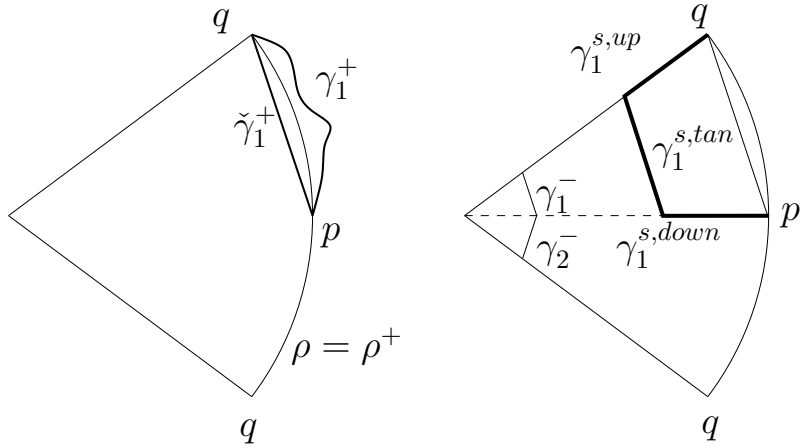
            					\end{center}
            					\caption{On the left, for $0\leq s\leq \frac{1}{8}$ the homotopy $H$ is a linear interpolation of the two curves $\gamma_1^+$ and $\check{\gamma}_1^+$. On the right, the thick piecewise linear line is $H (s, \cdot)$ for some (intermediate) $s\in [\frac{1}{8}, \frac{1}{4}]$.
            						At $s=\frac{1}{4}$ the curve $\gamma_1^{1,tan}$ coincides with $\gamma_1^-$. For $s\in [\frac{1}{4}, \frac{3}{4}]$ the homotopy keeps $\gamma_1^{1,down}$ and $\gamma_1^{1,up}$ fixed and ``swaps'' gradually $\gamma_1^-$ with $\gamma_2^-$. The fourth and fifth phases of the homotopy are then analogous, respectively, to the second and first.}\label{fig:omotopia}
            				\end{figure}
            				
            				Now, straight from the definitions one has, for $\varepsilon$ small enough, the trivial length estimates\footnote{While the curves $H(s,\cdot)$ are not always $C^1$ we still have a natural notion of length, gotten by means of piecewise linear approximations, which coincides with the usual one presented for $C^1$ curves.}
            				\[
            				L(H(s,\cdot))\leq
            				\begin{cases}
            				2(1+\varepsilon)\rho_{+}\sin\left(\frac{\pi}{2}\sin\alpha\right) & \textrm{for} \,  s\in\left[0,\frac{1}{8}\right] \sqcup \left[\frac{7}{8}, 1\right] \\
            				2(1+\varepsilon)\left((\rho_{+}-\rho_{s})+\rho_{s}\sin\left(\frac{\pi}{2}\sin\alpha\right)\right) & \textrm{for} \,  s\in\left[\frac{1}{8},\frac{1}{4}\right] \sqcup \left[\frac{3}{4},\frac{7}{8}\right] \\
            				2(1+\varepsilon)(\rho_{+}-\rho_{-})+C  & \textrm{for} \, s\in\left[\frac{1}{4},\frac{3}{4}\right] 
            				\end{cases}
            				\]
            				where $\rho_s$ stands for $(1-s)\rho_{+}+s\rho_{-}$ evaluated at $8(s-1/8)$ (respectively at $8(s-3/4)$) if $1/8\leq s\leq 1/4$ (respectively $3/4\leq s\leq 7/8$), and $C$ is a constant which does not depend on $\rho_{+}$. It follows that similar length estimates hold true for a suitable smoothing $\tilde{H}$ of $H$ possibly with a marginally worse multiplicative constant (say with $1+2\varepsilon$ in lieu of $1+\varepsilon$) and hence (by assuming, without loss of generality, constant speed parametrizations for the curves $\tilde{H}(s,\cdot), s\in [0,1]$) one can conclude
            				\[
            				\rho^{-2}_{+}\max_{s\in [0,1]} E\left(H\left(s,\cdot\right)\right)\leq 4(1+\varepsilon)^2+\rho^{-2}_{+}C 
            				\]
            				so that letting $\rho_{+}\to+\infty$ one obtains (getting back to the notation of the statement)
            				\[
            				\limsup_{r_0\to +\infty} \frac{\Lambda(r_0,\varphi_{0})}{r_0^2} \leq 4(1+\varepsilon)^2
            				\] 
            				and the arbitrariness of $\varepsilon$ allows to complete the proof.
            				
            			\end{proof}

            			We can now proceed with the proof of the following statement, ensuring the convergence of the sequences of solutions to the fixed-endpoints min-max problem studied above. 
            			
            			\begin{prop}\label{prop:existence} Let $(S,g)$ be an asymptotically conical surface of non-negative scalar curvature and, given a sequence $\left\{r_k\right\}$ with $r_k\nearrow +\infty$, let $p^{(r_k,\varphi_0)}, q^{(r_k,\varphi_0)}$  be a couple of antipodal points such that
            				\[
            				p^{(r_k,\varphi_0)}=q^{(r_k,\varphi_0)}=r_k, \ \varphi(p^{(r_k,\varphi_0)})=\varphi(q^{(r_k,\varphi_0)})-\pi=\varphi_0.
            				\]
            				Then, possibly extracting a subsequence, the geodesic segments $\Gamma_3^{(r_k,\varphi_0)}$ connecting $p^{(r_k,\varphi_0)}$ to $q^{(r_k,\varphi_0)}$ converge to a properly embedded geodesic line $\Gamma^{\left(\varphi_0\right)}_{\infty}$ of Morse index less or equal than one. If $(S,g)$ has positive scalar curvature, then equality holds.
            			\end{prop}	
            			
            			\begin{proof}
            				Let $(S,g)$ be the asymptotically conical surface in question, and let $C_\alpha$ be the corresponding asymptotic cone: by Theorem \ref{lem:pmt2} we can assume, without loss of generality, that $\alpha<\pi/2$ (otherwise the surface is flat $\mathbb{R}^2$ and the conclusion is trivial).
            				\
            				
            				\indent  For antipodal points $p^{(r_k,\varphi_0)}, q^{(r_k,\varphi_0)}$ as in the statement above, 
            				let $\gamma_i^{(r_k,\varphi_0)}: [0,1]\to S$ be constant speed parametrizations of the three gedesic constructed above (as per  Lemma \ref{lem:second} and Proposition \ref{pro:minmaxgeod}).
            				
            				\
            				
            				\textbf{Key claim 1:} there exists an open, bounded set $U\subset S$ such that $U\cap \Gamma^{(r_k,\varphi_0)}_3\neq\emptyset$ for a sequence of sufficiently large values of $k$.
            				
            				\
            				
            				Once this is proven the \textsl{first} conclusion of Proposition \ref{prop:existence} is straightforward, for the family of geodsics $\Gamma_3^{(r_k,\varphi_0)}$ for $r_k\geq\overline{r}_0$, having (patently) local curvature estimates and length estimates (for the latter see the argument below, \textbf{Key Claim 2}),  will converge to some embedded geodesic line $\Gamma^{(\varphi_0)}_{\infty}$ by Lemma \ref{lem:convmult}.
            				
            				\
            				
            				Therefore, we need to prove Key Claim 1. Arguing by contradiction, suppose it were false. That is to say, set for $r_k\geq \overline{r}_0$
            				\[
            				\tilde{r}_{k}=\sup\left\{r>0 \, : \ \ B_{r}(o)\cap \Gamma^{(r_k,\varphi_0)}_i=\emptyset \ \textrm{for} \ i=1,2,3 \right\}
            				\]
            				and assume that
            				\[
            				\sup_{k\geq 1}\tilde{r}_{k}=+\infty.
            				\]
            				(Here $o$ is the reference point defined in the previous proof).
            				It follows that given any value $\tilde{r}$ for $k$ large enough either the closed region bounded by $\Gamma_1^{(r_k,\varphi_0)}, \Gamma_3^{(r_k,\varphi_0)}$ is disjoint from $B_{\tilde{r}}(o)$ or the closed region bounded by $\Gamma_2^{(r_k,\varphi_0)}, \Gamma_3^{(r_k,\varphi_0)}$ is disjoint from $B_{\tilde{r}}(o)$. The argument is in fact identical in the two cases, so let us assume for the sake of definiteness to have to deal with the first one. Notice that we are not claiming that $\Gamma_1^{(r_k,\varphi_0)}$ and $\Gamma_3^{(r_k,\varphi_0)}$ only intersect at the endpoints, so that in particular the interior $\dot{D}^k$ of the region $D^k$ in question could consist of multiple connected components: in order to introduce a convenient notation let us set $\dot{D}^k=\sqcup_{i=0}^{d}\dot{D}^k_{i}$ for some $d\geq 0$ (this is well-defined for ODE uniqueness ensures that two distinct geodesics can only meet transversely, and at finitely many points). Let us first consider the case $d=0$: applying the Gauss-Bonnet theorem to the domain $D^k$ gives
            				\[
            				\int_{D^k} K_g+\nu^{ext}_{p^{(r_k,\varphi_0)}}+\nu^{ext}_{q^{(r_k,\varphi_0)}}       =2\pi
            				\]
            				where $\nu^{ext}_{p^{(r_k,\varphi_0)}}$ (resp. $\nu^{ext}_{q^{(r_k,\varphi_0)}}$) is the exterior angle between $\dot{\gamma}_1$ and $\dot{\gamma}_3$ at $p^{(r_k,\varphi_0)}$ (resp. at $q^{(r_k,\varphi_0)}$). The decay assumption on the metric implies that $ |K_g|(r,\varphi)\leq Cr^{-2-\mu}$ (where $\mu>0$ is the asymptotic decay rate of $(S,g)$ as per Definition \ref{def:acon})
            				and thus necessarily
            				\[
            				\lim_{k\to\infty} \int_{D^k} K_g=0
            				\]
            				because
            				\[
            				\int_{D^k} |K_g| \leq C\int_{\tilde{r}_k}^{r_k} r^{-2-\mu}r \,dr \leq C\int_{\tilde{r}_k}^{+\infty} r^{-2-\mu}r \,dr =\frac{C}{\mu}{\tilde{r}_k}^{-\mu}
            				\]
            				and thanks to the fact that $\tilde{r}_{k}\to+\infty$ as one lets $k\to\infty$, as remarked above.
            				
            				Hence we deduce that
            				\[
            				\lim_{r_0\to+\infty}\nu^{ext}_{p^{(r_k,\varphi_0)}}+\nu^{ext}_{q^{(r_k,\varphi_0)}}  =2\pi
            				\]
            				and since patently $\nu^{ext}_{p^{(r_k,\varphi_0)}}, \nu^{ext}_{q^{(r_k,\varphi_0)}}\in [0,\pi]$ we conclude that in fact
            				\[
            				\lim_{k\to\infty}\nu^{ext}_{p^{(r_k,\varphi_0)}} = \lim_{k\to\infty}\nu^{ext}_{q^{(r_k,\varphi_0)}}=\pi.
            				\]
            				This equation implies that when we rescale by a factor $r_k$ (and let $k\to\infty$), the geodesic $\Gamma_3^{(r_k,\varphi_0)}$ must subsequentially converge to $\Gamma'$ (recall that this operation is conformal). On the other hand, we know (from the proof of Proposition \ref{pro:minmaxgeod} where we constructed $\gamma_3^{(r_k,\varphi_0)}$ as geodesic attaining the min-max value) that
            				\[
            				\liminf_{k\to\infty}\frac{L(\gamma_3^{(r_k,\varphi_0)})}{r_k}\geq 2
            				\]
            				which is only possible if $\Gamma_3^{(r_k,\varphi_0)}$ converged locally to the radial geodesic $\Gamma'''$. These two facts determine the contradiction which completes the proof.	\\
            				
            				\indent
            				
            				The general case $d\geq 1$ follows along similar lines: 
            				\
            				 
            				\noindent  if we let $m^k_1,m^k_2,\ldots, m^k_{d}$ be the points of (transverse) interior intersection of  $\Gamma_1^{(r_k,\varphi_0)}$ and $\Gamma_3^{(r_k,\varphi_0)}$, as we move from $p^{(r_k,\varphi_0)}$ to $q^{(r_k,\varphi_0)}$, applying Gauss-Bonnet to the domain $D^k_i$ gives
            				\[
            				\int_{D^k_i} K_g =
            				\begin{cases}
            				2\pi - \nu^{ext}_{p^{(r_k,\varphi_0)}}-\nu^{ext}_{m^k_1} & \textrm{if} \ i=0 \\
            				2\pi -\nu^{ext}_{m^k_{i}}-\nu^{ext}_{m^k_{i+1}} & \textrm{if} \ i=1,\ldots, d-1 \\
            				2\pi -\nu^{ext}_{m^k_{d}}-\nu^{ext}_{q^{(r_k,\varphi_0)}} & \textrm{if} \ i=d.
            				\end{cases}
            				\] 
            				so that, adding these equation we get to
            				\[
            				\int_{D^k}K_g = 2\pi (d+1) -2\sum_{i=1}^{d}\nu^{ext}_{m^k_i}-\nu^{ext}_{p^{(r_k,\varphi_0)}}-\nu^{ext}_{q^{(r_k,\varphi_0)}}.
            				\]
            				Arguing as above, we know that the integral on the left-hand side must converge to zero as we let $k\to\infty$ and hence, once again \textsl{each} of the angles on the right-hand side must converge to $\pi$. In particular, this will force $\Gamma_3^{(r_k,\varphi_0)}$ to subconverge to $\Gamma'$ as we rescale by $r_k$, which violates the gap condition above.  
            				
            				\
            				
            				Such claim being proven, we need to gain local length estimates for the sequence of geodesics $\Gamma^{(r_k,\varphi_0)}_3$. 
            				
            				\
            				
            				\textbf{Key claim 2:} for every $v\in S$ there exists a metric ball $B_{\delta}(v)$ and a constant $C>0$ such that $\mathscr{H}^{1}(\Gamma^{(r_k,\varphi_0)}_3\cap B_{\delta}(v))\leq C$ for all $k\geq 1$.
            				
            				\
            				
            				If the claim were false, we would have concentration of length in some bounded region of $S$. In particular, we could find (without loss of generality):
            				\begin{enumerate}
            					\item{a geodesic segment $\Gamma$ and a tubular neighborhood thereof having the form $\Gamma\times(-\delta,\delta)$;}	
            					\item{a set of suitable coordinates $\left\{x\right\}$ on such tubular neighborhood, so that $\Gamma$ is defined by $-\eta<x_1<\eta, \ x_2=0$;}
            					\item{for a subsequence of large radii $r_k$ (at least) two smooth functions $f^{k}_1, f^{k}_2: (-\eta,\eta)\to(-\delta,\delta)$ such that $f^k_1<f^k_2$ and $\textrm{graph}(f^k_1)\cup \textrm{graph}(f^k_2)\subset \Gamma^{(r_k,\varphi_0)}_3$, furthermore both $f^k_1$ and $f^k_2$ converge smoothly to zero as we let $k\to\infty$.}	
            				\end{enumerate}	
            				In such case consider the compact region $\Omega_k$ bounded by a short geodesic segment connecting (almost orthogonally) $f^k_1(0)$ with $f^k_2(0)$ together with a segment of $\Gamma^{(r_k,\varphi_0)}_3$, see Figure \ref{fig:dito}. By the Gauss-Bonnet Theorem we then must have 
            				\[
            				\lim_{k\to\infty}\int_{\Omega_k}K_g=\pi\, .
            				\]
            				Next, since the rescalings of $\Gamma_3^{(r_k, \varphi_0)}$ by a factor $r_k$ are converging to $\Gamma'''$ (the convergence being smooth and uniform in the coordinate annulus of radii $1/3$ and $1$), we conclude that
            				$\Gamma_3^{(r_k, \varphi_0)}$ meets the circle $\{r=\frac{1}{2} r_k\}$ at two almost antipodal points at almost square (exterior) angles $\nu^{'}_k, \nu^{''}_k $, cf. again Figure \ref{fig:dito}. Denoted $\Sigma^{+}_k, \Sigma^{-}_k$ the two resulting connected components of $\{r = r_k/2\}$,  we easily conclude that
            				\[
            				\lim_{k\to\infty}\int_{\Sigma^{+}_k}\kappa_g=\lim_{k\to\infty}\int_{\Sigma^{-}_k}\kappa_g=\pi \sin \alpha
            				\]
            				as well as
            				\[
            				\lim_{k\to\infty}\nu^{'}_k=\lim_{k\to\infty}\nu^{''}_k=\frac{\pi}{2}.
            				\]
            				Thus, applying the Gauss-Bonnet theorem to the two regions $\Omega^{+}_{k}, \Omega^{-}_k$ bounded by $\Sigma^{+}$ (resp. $\Sigma^{-}_k$) together with $\Gamma^{(r_k,\varphi_0)}_3$ we get at once
            				\[
            				\lim_{k\to\infty}\int_{\Omega^{+}_k}K_g=\lim_{k\to\infty}\int_{\Omega^{-}_k}K_g=\pi (1-\sin \alpha).
            				\]
            				
            				\begin{figure}[htbp]
            					\begin{center}
            						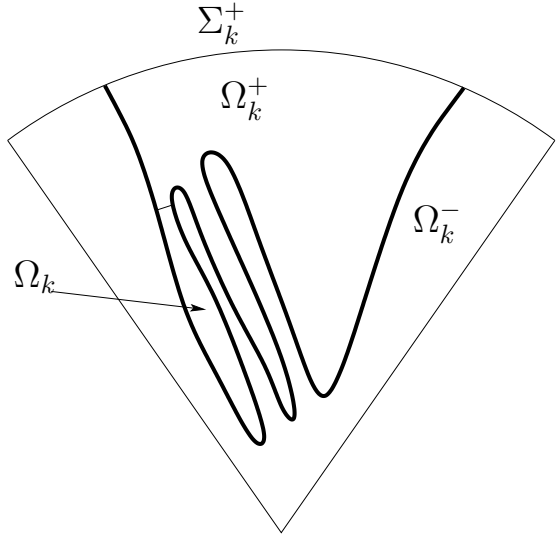
            					\end{center}
            					\caption{The regions $\Omega_k^\pm$ and $\Omega_k$. The thick line represents the geodesic $\Gamma_3^{(r_k, \varphi_0)}$.}\label{fig:dito}
            				\end{figure}
            				
            				But on the other hand, the region $\Omega_k$ is contained in either $\Omega^{+}_k$ or $\Omega^{-}_k$, which leads to a contradiction because $K_g$ is non-negative. Thus we conclude that the local concentration of min-max geodesics cannot occur.

            				\
            				
            				Lastly, let us discuss the index of the properly embedded geodesic $\Gamma_{\infty}$. 
            				Now, it is a direct, straightforward check that in fact $\gamma^{(r_0,\varphi_0)}_3$ has also index less or equal than one as a critical point of the length functional $L(\cdot)$ (for if not we could reparametrize two $W^{1,2}$-orthogonal variations that decrease the length into variations by constant speed in which category $E=L^2$ and so those would be two  variations that decrease the energy, contradiction). Hence, $\Gamma^{(\varphi_0)}_{\infty}$ does also have Morse index less or equal than one due to locally graphical, smooth (geometric) convergence of $\Gamma^{(r_0,\varphi_0)}_3$ to $\Gamma^{(\varphi_0)}_{\infty}$ with multiplicity one. 
            				
            				Let us finally concern ourselves with the rigidity part of our theorem.
            				If $\gamma^{(\varphi_0)}_{\infty}: \mathbb{R}\to S$, an arclength parametrization for $\Gamma^{(\varphi_0)}_{\infty}$, had index zero then the stability inequality for geodesics
            				\[
            				\int_{-\infty}^{+\infty}|\dot{u}(t)|^2\,dt\geq \int_{-\infty}^{+\infty}K(\gamma_{\infty}(t))u^2(t)\,dt
            				\]
            				applied to a cutoff function
            				\[
            				u^{(\sigma)}(t)=
            				\begin{cases}
            				1 & \textrm{if} \ |t|\leq\sigma \\
            				0 & \textrm{if} \  |t|\geq 2\sigma
            				\end{cases}
            				\]
            				and satisfying $|\dot{u}^{(\sigma)}|\leq 2/\sigma$, implies
            				\[
            				\int_{-\sigma}^{+\sigma}K(\gamma^{(\varphi_0)}_{\infty}(t))\,dt\leq \frac{16}{\sigma}
            				\]
            				so that letting $\sigma\to+\infty$ we conclude that the Gauss curvature must vanish identically along $\Gamma^{(\varphi_0)}_{\infty}$.

            			\end{proof}

            			\section{Back to the question of Bangert}\label{sec:bangert}
            			
            			In this section, we shall complete the proof of Theorem \ref{thm:maingeod} by showing that indeed the map  $[\varphi_0]\to \Gamma^{(\varphi_0)}_{\infty}$ is injective (for $[\varphi]\in\mathbb{R}\mathbb{P}^1$ the equivalence class of $\varphi_0\in S^1$ modulo antipodality), so that we will in fact obtain \textsl{uncountably many} embedded geodesic lines. Such conclusion follows at once from a \textsl{non-twisting} statement we are about to present. To that aim we need a definition and a simple lemma.
            			
            			
            			\begin{defi}\label{def:ray}
            				Given a properly embedded line $\Gamma\subset S$, we call \textsl{ray} a closed, unbounded connected component of $\Gamma$. 
            			\end{defi}	
            			In particular, it follows that a ray can be parametrized by means of a map $\gamma: [0,+\infty)\to S$.	
            			
            			\begin{lem}\label{lem:asygraph}
            				Let $(S,g)$ be an asymptotically conical surface (in the sense of Definition \ref{def:acon}) and let $\gamma:[0,+\infty)\to S$ be a properly embedded geodesic ray. Then, one can find an identification of $S$ minus a suitable compact set with the outer portion of a Euclidean wedge of angle $2\pi\sin\alpha$ (with pointwise identification of the sides, see Remark \ref{ref:wedgecoor2}), a large constant $\rho_0>0$ and a defining function $f\in C^2([\rho_0,+\infty);\mathbb{R})$ whose Cartesian graph\footnote{The wedge $W$ in question has natural Euclidean coordinates $(x^1,x^2)$, where of course $x^1=\rho\cos\vartheta, x^2=\rho\sin\vartheta$, and so the Cartesian graph of a function $f:I\to\mathbb{R}$ is meant to be the set $\left\{(x^1,x^2)\in W, x^2=f(x^1) \ \textrm{for} \ x^1\in I\right\}$.}coincides with the image $\gamma[0,+\infty)\cap \left\{\rho\geq \rho_0\right\}$ and such that $|f(\rho)|+\rho|f'(\rho)|\leq C \rho^{1-\mu}$ (for $\mu>0$ the asymptotic decay rate of the surface in question). 	
            			\end{lem}

            			This assertion can be proved by observing that the geometric assumption of vanishing geodesic curvature, namely $\kappa_g=0$, implies $\kappa_{\delta}=O(\rho^{-1-\mu})$ and hence noticing that such decay rate (by integrability of $\rho\mapsto \rho^{-1-\mu}$ when $\mu>0$) ensures uniqueness of the tangent cone at infinity of $\textrm{spt}(\gamma)$ and hence the indefinite extension of a local graphical description of such support, with the claimed expansion. This is a (simpler) variation of well-known arguments for complete minimal surfaces, cf. \cite{Car14}, so we omit the details. \\
            			
            			\indent In case the conclusion of Lemma \ref{lem:asygraph} holds we shall say that the curve $\gamma$ is asymptotic to the coordinate half-line $\vartheta=0$ (obviously, this is to be understood in a suitably weak sense if $0<\mu\leq 1$). Now, the claimed non-twisting phenomenon can be phrased as follows:
            			
            			\begin{prop}\label{prop:multiplicity}
            				Let $(S,g)$ be an asymptotically conical surface (in the sense of Definition \ref{def:acon}) and let $\gamma:[0,+\infty)\to S$ be a properly embedded geodesic ray, asymptotic to the half-line $\vartheta=0$. For any fixed $0<\vartheta<\pi\sin\alpha$ one cannot find two diverging sequences of radii $\left\{\rho_k\right\}$ and $\left\{\rho'_k\right\}$ with $\rho'_k<\rho_k$ for all $k\geq 1$ and a sequence of solutions $\left\{\Gamma_k\right\}$ to the min-max Plateau problems with endpoints $p_k, q_k$ with $\rho(p_k)=\rho(q_k)=\rho_k, \vartheta(q_k)=\vartheta, \vartheta(p_k)=\vartheta-\pi\sin(\alpha)$
            				such that $\Gamma_k\supset \textrm{graph}(f_k)$ with $\sup_{\rho'_k/2\leq\rho\leq 2\rho'_k}\left(|f_k(\rho)|+\rho'_k|f'_k(\rho)|\right)<\rho'_k k^{-1}$.
            			\end{prop}	
            			
            			\begin{proof}
            				Let us argue by contradiction, assuming the existences of scales and min-max geodesics as in the statement above. For each $k\geq 1$ consider an arc-length parametrization $\gamma_k:[s^{-}_k,s^{+}_k]$ of $\Gamma_k$ such that $m_k:=\gamma_k(0)\in\textrm{graph}(f_k)\cap \left\{\rho=\rho'_k\right\}$ and $\frac{d\rho(\gamma_k)}{ds}_{s=0}>0$. Basic Morse-theory (which amounts, in the special case of curves, to a direct curvature comparison using large coordinate circles) ensures that in fact
            				\[
            				\frac{d\rho(\gamma_k)}{ds}\geq 0 \ \forall \ s\in [0,s^{+}_{k}]
            				\]
            				which means that the radius function (when restricted to $\Gamma_k$) is monotone non-decreasing from $m_k$ to $q_k$. Set $\Gamma^{\omega}_k = \gamma_k[0,s^{+}_k]$ and $\Gamma^{\lambda}_k=\Gamma \cap\left\{\rho'_k\leq\rho\leq\rho_k\right\}$. Notice that (possibly neglecting finitely many terms in the sequence and renaming indices) we can assume that $\Gamma^{\lambda}_k$ is the (Cartesian) graph of a defining function $f$ restricted to $[\rho_k',\rho_k]$ and satisfying the bounds described in Lemma \ref{lem:asygraph}.

            				Let us then consider the (possibly multiply-connected) domain $D^k$ whose piecewise smooth boundary consists of of $\Gamma^{\omega}_k, \Gamma^{\lambda}_k$ and the arcs of coordinate circles at radii $\rho'_k$ and $\rho_k$, which we shall call $\Delta'_k$ and $\Delta_k$, respectively. 	 
            				Set $\dot{D^k}=\sqcup_{i=0}^d\dot{D^k}_i$ where $m_k\in D^k_0$ and $q_k\in D^k_d$ (as shown in Figure \ref{fig:domains}).
            				
            				\begin{figure}[htbp]
            					\includegraphics{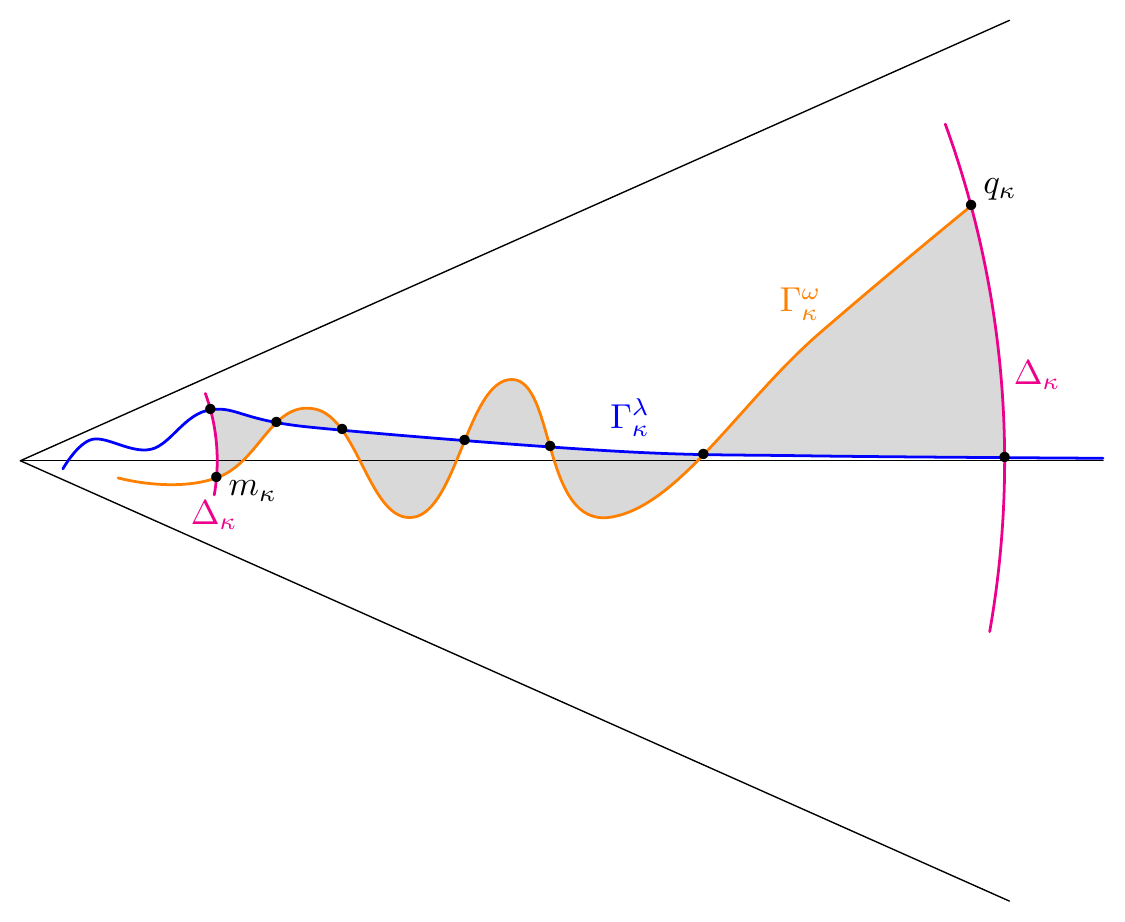}
            					\caption{The twisting phenomenon: we rule out the existence of min-max geodesic segments that behave like the orange line (by courtesy of Mario B. Schulz).}
            					\label{fig:domains}
            				\end{figure}

            				At this stage, let us simply apply the Gauss-Bonnet theorem to each domain $D^k_i$ for $i=0,1,\ldots, d$. Let $\nu^k_1,\ldots,\nu^k_d$ the exterior angles at the intersection points of $\Gamma^{\omega}_k$ and $\Gamma^{\lambda}_k$ (if there is no such intersection point the proof is identical and in fact simpler). One has that
            				\[
            				\int_{D^k_i} K_g = o(1) \ \ \forall \ i=0,1,\ldots,d \ \textrm{as one lets} \ k\to\infty
            				\]
            				by virtue of the integrability of the Gauss-curvature function $K_g$ (which, in turn, is implied by the bound $|K_g|\leq C\rho^{-2-\mu}$).
            				Furthermore, for what concerns the integral of the geodesic curvature along the boundary
            				\begin{equation*}
            				\int_{\partial D^k_i} \kappa_g
            				\begin{cases}
            				\leq C((\rho'_k)^{-\mu}+k^{-1}) & \textrm{if} \ i=0  \\
            				= 0 & \textrm{if} \  0<i<d \\
            				= \vartheta(1+o(1)) & \textrm{if} \ i=d. \\
            				\end{cases}	
            				\end{equation*} Lastly, the exterior angles at the four intersection points $\Gamma^{\lambda}_k\cap \Delta'_k, \Gamma^{\omega}_k\cap \Delta'_k, \Gamma^{\lambda}_k\cap \Delta_k, \Gamma^{\omega}_k\cap \Delta_k$ are all $\pi/2+o(1)$ as we let $k\to\infty$. (Notice that, in the case of the angle at $\Gamma^{\omega}_k\cap \Delta_k$ this is a consequence of the blow-down characterization of the min-max segments we construct, which in turn is directly implied by Proposition \ref{pro:minmaxgeod}. Thus, possibly at the cost of extracting a subsequence we can always ensure that this angle converges to $\pi/2$ as well). As a result, proceeding inductively for $i=0,1,\ldots, d-1$ the Gauss-Bonnet theorem provides $\nu^k_i\to \pi$ for $i=1,2,\ldots, d$ as $k\to 0$ and hence for $i=d$
            				\[
            				2\pi= \int_{D^k_d}K_g+\int_{\partial D^k_d}\kappa_g+\textrm{exterior angles}= o(1)+2\pi+\int_{\Delta_k}\kappa_g
            				\]
            				so that one should conclude $\vartheta (1+o(1))= \int_{\Delta_k}\kappa_g=o(1)$ which gives the desired contradiction as soon as one lets $k\to\infty$.
            			\end{proof}


            			\appendix
            			
            			\section{Geodesics, 1-currents and convergence results}\label{sec:geodcurr}
            			
            			\subsection{Geodesics}\label{subs:geod}
            			Let $(N,g)$ be a complete, Riemannian manifold of dimension greater or equal than two. We will say that a $C^2$-curve $\gamma:I\to N$ is a \ul{parametrized geodesic} if $D_{\dot{\gamma}}\dot{\gamma}=0$ where $D$ denotes the Levi-Civita connection on $(N,g)$, the apex $\dot{}$ denotes ordinary differentiation with respect to the parameter and $I\subset \mathbb{R}$ is an interval. If $I=[a,b]$, a compact interval, it is well-known that  $\gamma$ (as above) is a geodesic if and only if it is a critical point of the length functional
            			\[
            			L(\gamma)=\int_I \sqrt{g(\dot{\gamma}(t),\dot{\gamma}(t))}\,dt.
            			\]
            			The same characterization also holds true in general (hence, for instance, when $I=\mathbb{R}$) for variations that are supported on relatively compact subdomains of $I$.
            			A posteriori, a geodesic is in fact a smooth curve, namely $\gamma\in C^{\infty}(I,N)$.\\ 
            			\indent It is often convenient to work with the energy functional
            			\[
            			E(\gamma)=\int_I g(\dot{\gamma}(t),\dot{\gamma}(t))\,dt
            			\]
            			for which the Cauchy-Schwarz inequality gives $L^2\leq E|I|$. In particular, if $I=[0,1]$ and $\gamma$ is parametrized by a constant multiple of the arc-length then $L^2=E$. Hence it is easily seen that a critical point of $E$ is also a critical point of $L$ and, viceversa, a critical point of $L$ can be re-parametrized so to become a critical point of $E$. \\
            			\indent If we set $\Gamma:=\gamma(I)$ then one can canonically associate to $\Gamma$ an integral 1-current $T$ (with unit multiplicity and orientation induced by the parametrization itself) and of course
            			$\textrm{spt}(T)=\Gamma$ while $\textrm{spt}(\partial T)$ consist of the endpoints of $\Gamma$.
            			Notice that
            			(assuming, say, compactness of $I$) one has $L(\gamma)=\mathscr{H}^{1}(\Gamma)$. If $\gamma:I\to N$ is a parametrized geodesic, we shall say (with slight abuse of terminology) that $\Gamma$ is a geodesic (rather than the \ul{support} of a geodesic). This choice, which we adopt for the sake of brevity, is justified by the basic fact that for every diffeomorphism $\lambda: I_1\to I$ one has that $\gamma: I\to N$ is a geodesic if and only if $\gamma\cdot \lambda: I_1\to N$ is.
            			
            			\subsection{Convergence}\label{subs:convergence}
            			
            			Geodesics could also, obviously, be regarded as the one-dimensional, degenerate counterpart of minimal surfaces and this analogy suggests the effectiveness of dealing with convergence of supports, rather than parametrizations. Notice that, in fact, geodesics are the one-dimensional counterpart of totally geodesic surfaces so that the corresponding curvature estimates come (tautologically) for free. 
            			
            			\begin{defi}\label{conv}
            				Let $(N,g)$ be a Riemannian manifold of dimension two and let $\left\{\Gamma_k\right\}_{k\geq 1}$ be a sequence of smooth, connected 1-dimensional submanifolds (possibly with boundary). We shall say that such sequence converges \ul{geometrically} with multiplicity $m\geq 1$ if there exists a smooth 1-dimensional submanifold $\Gamma$ such that:
            				\begin{itemize}
            					\item{for every point $p$ of $\Gamma\setminus \partial \Gamma$ one can find an open tubular neighborhood $U$ and local coordinates $\left\{x\right\}$ such that $\Gamma\cap U$ is described by the equation $x_{2}=0$ and for $k\geq k_0$ the support $\Gamma_k$, when restricted to $U$, consists of exactly $m$ smooth graphs, namely if $U=(-\delta_1,\delta_1)\times(-\delta_2,\delta_2)$ there exist  $f_i\in C^{\infty}((-\delta_1,\delta_1),\mathbb{R})$ with $f_1<f_2<\ldots<f_m$ so that
            						\[
            						\Gamma_k = \left\{(x_1, x_2) \in(-\delta_1,\delta_1)\times(-\delta_2,\delta_2) \ : \ x_2= f_i(x_1), \ i=1,2,\ldots, m \right\}
            						\]
            						and each function $f_i$
            						converges to zero in $C^{\infty}$ as we let $k\to\infty$;}	
            					\item{if $\partial \Gamma$ is not empty, then $\partial \Gamma_k=\partial\Gamma$ (at least for $k\geq k_0$) and the above condition holds with $m=1$ both for interior and, with straightforward modifications, for boundary points.}	
            				\end{itemize}	 	
            			\end{defi}	
            			
            			We mention here two simple compactness results that are frequently used in this paper.
            			
            			\begin{lem}\label{lem:convmult}
            				Let $(N,g)$ be a Riemannian manifold of dimension two and let $\left\{\Gamma_k\right\}_{k\geq 1}$ be a sequence of smooth, simple geodesics with locally uniform length bounds, namely assume that for every $p\in N$ there exists a bounded regular neighborhood $U$ such that 
            				$\mathscr{H}^1(\Gamma_k\cap  U)\leq C$ for a constant $C>0$ that is independent of $k$.
            				If $\left\{\Gamma_k\right\}$ does not escape from every bounded domain of $N$, then there exists a smooth geodesic $\Gamma$ such that, possibly extracting a subsequence (which we will not rename),  $\left\{\Gamma_k\right\}$ converges geometrically to $\Gamma$. 
            			\end{lem}	
            			
            			In presence of a non-empty boundary, one can gain sub-convergence with unit multiplicity.

            			\begin{lem}\label{lem:convint}
            				Let $(N,g)$ be a Riemannian manifold of dimension two and let $\left\{\Gamma_k\right\}_{k\geq 1}$ be a sequence of smooth geodesic segments all sharing one endpoint and with uniformly bounded length.
            				Then there exists a smooth geodesic $\Gamma$ such that, possibly extracting a subsequence (which we will not rename),  $\left\{\Gamma_k\right\}$ converges geometrically to a geodesic segment $\Gamma$ with multiplicity one. 
            			\end{lem}

            			\bibliographystyle{plain}

            		\end{document}